\documentclass[11pt,a4paper,twoside]{amsart}
\usepackage{amsmath,amsfonts,amsthm,amsopn,color,amssymb,enumitem}
\usepackage{palatino}
\usepackage{graphicx}
\usepackage[colorlinks=true]{hyperref}
\hypersetup{urlcolor=blue, citecolor=red, linkcolor=blue}

\usepackage{relsize}
\usepackage{esint}

\newenvironment{pfn}{\noindent{\em Proof}}{\rule{2mm}{2mm}\medskip}

\newcommand{\e}{\varepsilon}

\newcommand{\R}{\mathbb{R}}

\newcommand{\weakto}{\rightharpoonup}

\DeclareMathOperator{\supp}{supp}

\renewcommand{\a }{\alpha }

\renewcommand{\b }{\beta }

\renewcommand{\d }{\delta }

\newcommand{\g }{\gamma }

\renewcommand{\l }{\lambda}
\newcommand{\n }{\nabla }

\newcommand{\G}{\Gamma}

\renewcommand{\S}{\Sigma}

\newcommand{\N}{\mathbb{N}}

\def\bbm[#1]{\mbox{\boldmath $#1$}}
\newcommand{\beq }{\begin{equation}}
\newcommand{\eeq }{\end{equation}}

\newcommand{\dis}{\displaystyle}

\newtheorem{theorem}{Theorem}[section]
\newtheorem{lemma}[theorem]{Lemma}

\newtheorem{proposition}[theorem]{Proposition}
\newtheorem{remark}[theorem]{Remark}

\title[Conformal metrics with prescribed gaussian and geodesic curvatures]
{Conformal metrics with prescribed gaussian and geodesic curvatures}

\author{Rafael L\'{o}pez-Soriano}
  \address{Rafael L\'{o}pez-Soriano \\
    Universitat de Val\`encia\\
    Departamento de An\'alisis Matem\'atico\\
    Dr. Moliner 50 \\
    46100 Burjassot (Valencia), Spain.}
  \email{rafael.lopez-soriano@uv.es}

\author{Andrea Malchiodi}
  \address{Andrea Malchiodi \\ 
Scuola Normale Superiore\\
Piazza dei Cavalieri, 7\\
56126, Pisa, Italy.}
  \email{andrea.malchiodi@sns.it}

\author{David Ruiz}
  \address{David Ruiz \\
    Universidad de Granada\\
    Departamento de An\'alisis Matem\'atico\\
    Campus Fuentenueva\\
    18071 Granada, Spain}
  \email{daruiz@ugr.es}

\thanks{R. L.-S. and D. R. have been supported by the FEDER-MINECO Grant MTM2015-68210-P and by J. Andalucia (FQM-116). A.M. has been supported by the project {\em Geometric Variational Problems} and {\em Finanziamento a supporto della ricerca di base} from Scuola Normale Superiore and by MIUR Bando PRIN 2015 2015KB9WPT$_{001}$.  He is also member of GNAMPA as part of INdAM.     }



\begin{document}

\begin{abstract}
We consider the problem of prescribing the Gaussian and the geodesic curvatures of a compact surface with boundary by a conformal deformation of the 
metric. 
We derive some existence results using a variational approach, either by minimization of 
the Euler-Lagrange energy or via min-max methods. One of the main tools in our approach 
is a blow-up analysis of solutions, which in the present setting can have diverging volume. To our 
knowledge, this is the first time in which such an aspect is treated.  Key ingredients in our arguments 
are: a blow-up analysis around a sequence of points different from local maxima; the use of holomorphic domain-variations; and Morse-index estimates.

\end{abstract}

\maketitle

\section{Introduction}
\setcounter{equation}{0}

A classical problem in Geometry is the prescription of  the Gaussian curvature on a compact Riemannian surface $\Sigma$ under a conformal change of the metric, dating back to \cite{Ber, KazWar}. Denote by $\tilde{g}$ the original metric, by $g$ the conformal one, and by $e^u$ the conformal factor (that is, $g = e^u \tilde{g}$). The curvature then transforms according to the law: 

$$ - \Delta u + 2 \tilde{K}(x) = 2 K(x)e^{u},$$
where $\Delta=\Delta_{\tilde{g}}$ stands for the Laplace-Beltrami operator associated to the metric $\tilde{g}$, and $\tilde{K}$, $K$ stand for the Gaussian curvatures with respect to $\tilde{g}$ and $g$, respectively. The solvability of this equation has been studied for several decades, and it is not possible to give here a comprehensive list of references.  We refer the interested reader to Chapter 6 in the book \cite{aubin-book}.

If $\Sigma$ has a boundary, other than the Gaussian curvature in $\S$ it is natural to prescribe   also the geodesic curvature on $\partial \S$. 
Denoting by $\tilde{h}$ and $h$ the geodesic curvatures of the boundary with respect to $\tilde{g}$ and $g$ respectively, we 
are led to the  boundary value problem:

\begin{equation}\label{ecua0}
\left\{\begin{array}{ll}
\displaystyle{-\Delta u +2 \tilde{K}(x)= 2 K(x) e^u},  \qquad & \text{in $\Sigma$,}\\
\displaystyle{\frac{\partial u}{\partial n} + 2 \tilde{h}(x) = 2h(x)e^{u/2}}, \qquad  &\text{on
$\partial\Sigma $.}
\end{array}\right.
\end{equation}

In the literature there are results on some versions of the latter  problem. For example, the case $h=0$ has been treated   in \cite{chang1}, while the case $K=0$ has been considered in \cite{chang2, li-liu, liuwang}. There is also work on the blow-up analysis of solutions, see \cite{bao, francesca}, although the phenomenon is not fully understood. 

The case of constant $K$, $h$ has also been considered: in \cite{brendle} the author used a parabolic flow to obtain solutions in the limit.  Using methods from complex analysis  and the structure of Liouville equations, explicit expressions for the solutions and the exact values of the constants were determined if $\Sigma$ is a disk or an annulus, see \cite{otro, Asun}. Some classification results for the half-plane are also available in \cite{mira-galvez, li-zhu, zhang}. However, there are almost no results for 
the general situation in which both curvatures are variable functions: the following ones, are the only we are aware of at the moment. In \cite{cherrier} partial existence results are given, but some of them  include an undetermined Lagrange multiplier. A Kazdan-Warner  obstruction to existence has been found in \cite{hamza}. Very recently, a result for positive symmetric curvatures in the disk has appeared, see \cite{sergio}.  

In higher dimensions the natural analogous question regards the simultaneous prescription of the 
scalar curvature and the boundary mean curvature. The scalar-flat case with  constant mean curvature is known as the \emph{Escobar problem}, in strong relation with the Yamabe problem. In this regard, see \cite{amb, dja, escobar Annals, escobar Indiana, felli, HanLi1, HanLi2, marques} and the references therein.

Integrating \eqref{ecua0} and applying the Gauss-Bonnet theorem, one obtains

\begin{equation} \label{GB}
\int_{\Sigma}K e^u +  \oint_{\partial \Sigma}h e^{u/2}=  2\pi\chi(\Sigma),
\end{equation}
where $\chi(\S)$ is the Euler characteristic of $\S$.

In this paper we study existence and compactness of solutions of \eqref{ecua0} in the negatively-curved case, namely when $K(x)<0$. For existence, we focus on the case $\chi(\Sigma) \leq 0$. The case of the disk is intrinsically more complicated due to the non-compact effect of the group of M\"obius  maps, and it will be studied in a future work. However one of the main goals of the paper is the blow-up analysis given in Theorem \ref{teo4}, which is very general and applies to any PDE in the form \eqref{ecua0}.

It is easy to see that, via a conformal change of metric, we can always prescribe the values $h=0$, $K= sgn \chi(\Sigma)$, see Proposition  \ref{easy}. Hence, without loss of generality, we can assume that our initial metric is such that $\tilde{K}$ is constant and $\tilde{h}=0$. The problem to study  then becomes:

\begin{equation}\label{ecua}
\left\{\begin{array}{ll}
\displaystyle{-\Delta u +2 \tilde{K}= 2 K(x) e^u},  \qquad & \text{in $\Sigma$,}\\
\displaystyle{\frac{\partial u}{\partial n} = 2h(x)e^{u/2}}, \qquad  &\text{on
$\partial\Sigma $,}
\end{array}\right. 
\end{equation}
where $\tilde{K} = sgn \, \chi(\Sigma)$. 

If $K(x)<0$, problem \eqref{ecua} is the Euler-Lagrange equation of the energy functional $I: H^1(\Sigma) \to \R$, 

\begin{equation} \label{functional} I(u)= \int_{\Sigma} \left( \frac 1 2 |\nabla u|^2 + 2 \tilde{K} u + 2 |K(x)| e^u \right) 
- 4 \oint_{\partial \Sigma} h(x) e^{u/2}.\end{equation}

Observe that the area and  boundary terms are in competition, and a priori it is not clear whether $I$ is bounded from below or not. For the statement of our results it will be convenient to define the function $\mathfrak{D}: \partial \S \to \R$ as 

\begin{equation} \label{d} \mathfrak{D}(x)= \frac{h(x)}{\sqrt{|K(x)|}}.\end{equation}

Notice that $\mathfrak{D}$ is scale-invariant. As we shall see, the function $\mathfrak{D}$ will play a crucial role in the global behavior of the functional $I$, as well as in the blow-up analysis of solutions to \eqref{ecua0}.

Our first existence result deals with the case $\chi(\S)<0$.

\begin{theorem}\label{teo1} Assume that  $\tilde{K}<0$. Let $K$, $h$ be continuous functions such that $K<0$ and $\mathfrak{D}(p) < 1$ for all $p \in \partial \Sigma$. Then $I$ attains its infimum and hence problem \eqref{ecua} admits a solution. If moreover $h \leq 0$, then the solution is unique. \end{theorem}

The next theorems address the case $\chi(\S)=0$.

\begin{theorem} \label{teo2} Assume that  $\tilde{K}=0$. Let $K$, $h$ be continuous functions such that $K<0$ and

\begin{enumerate}
\item $\mathfrak{D}(p) < 1$ for all $p \in \partial \Sigma$;
\item $\oint_{\partial \Sigma} h(x)  >0$.
\end{enumerate}

Then $I$ attains its infimum and hence problem \eqref{ecua} admits a solution. \end{theorem}

Compared to Theorem \ref{teo2}, our next result is concerned with the reversed case of the inequality 
for $\mathfrak{D}$ and $\oint_{\partial \Sigma} h(x)$.

\begin{theorem} \label{teo3} Assume that  $\tilde{K}=0$. Let $K$, $h$ be $C^1$ functions such that $K<0$ and

\begin{enumerate}
\item $\mathfrak{D}(p)>1$ for some $p \in \partial \Sigma$;
\item $\oint_{\partial \Sigma} h(x) <0$;
\item $\mathfrak{D}_{\tau}(p) \neq 0$ for any $p \in \partial \Sigma$ with $\mathfrak{D}(p)=1$.
\end{enumerate}

Then problem \eqref{ecua} admits a solution. \end{theorem}

We point out that the integration appearing in {\em (2)}, for both Theorem \ref{teo2} and Theorem \ref{teo3}, is done 
with respect to the (unique) conformal metric with minimal boundary obtained in Proposition \ref{easy}. If one 
wishes to go back to the original formulation \eqref{ecua0}, these conditions should be properly rephrased. 

We shall see also that the hypotheses of the above theorems are somehow natural. In the case of the annulus, for instance, the classification of  solutions for constant curvatures given in \cite{Asun} matches with our results. In this regard we shall also show an obstruction result for the existence of solutions, using an argument from \cite{Ros}. See Section 2 for details.

We shall show that under the assumptions of Theorem \ref{teo3} $I$ is not bounded from below. In this case we use a min-max argument to find a saddle-type solution to the problem. However the Palais-Smale property - i.e., convergence of approximate solutions - is not known to hold for the functional $I$. We bypass this problem via the well-known \emph{monotonicity trick} by Struwe; roughly speaking, we are able to prove existence if compactness of solutions to perturbed problems is guaranteed.

\medskip This motivates the study of blowing-up solutions. The question is: given a sequence $u_n$ of solutions of problems in the form \eqref{ecua0}, are they bounded from above? Assuming otherwise, we may define the \emph{singular set } as

\begin{equation} \label{singularset} S=\{p \in \Sigma: \ \exists \ x_n \to p \mbox{ such that } u_n(x_n) \to +\infty\}, \end{equation}
see \cite{breme}. Starting with \cite{breme, li-sha}, the properties of such blowing-up solutions for Liouville-type equations have been much studied in the literature. In sum, near any isolated point $p \in S$ one can rescale the solution and obtain in the limit an entire solution, and then take advantage of classification results. In this framework, a finite mass condition is always assumed, as for instance:

$$ \int_{\S} |K_n(x)| e^{u_n} < + \infty.$$ 
In the previous literature this condition is essential in many aspects, which we enumerate below.

\begin{enumerate}
\item The entire solutions of the limit problem (in the plane or the half-plane) are much more restricted in the case of finite mass. For instance, the classification of the solutions to the equation

$$ -\Delta U = 2 e^U \ \mbox{ in } \R^2,$$
dates back to Liouville (\cite{Liouville}), and form a large family of non-explicit solutions. However, under the finite mass assumption, the only solutions are given by:

$$ U(x)= 2 \log \Big ( \frac{2 \lambda}{1 + \lambda^2 |x-x_0|^2}\Big ),\ \ x_0 \in \R^2, \ \lambda >0, $$
see \cite{chen-li}.

\item  Finite total mass implies also that the set $S$ is finite, since $e^u$ behaves as a finite combination of Dirac deltas 
with weights bounded away from zero (see \cite{breme}). The use of the  Green's representation formula gives then some global information on the behavior of the solutions.

\item Since $S$ is finite, one has local maxima of $u_n$ around each of the point of $S$, and one can pass to the limit after a suitable rescaling.

\end{enumerate}

However, in our problem, blowing-up solutions may have diverging mass. Observe that \eqref{GB} allows some compensation between the area and the boundary terms, which can both diverge. We would like to emphasize that this is not a technical issue: indeed, those masses will diverge in some cases. This is a completely new phenomenon and one of the main goals of the paper is to give a complete description of it. Our main result in this respect is the following:

\begin{theorem} \label{teo4}
Let $u_n$ be a blowing-up sequence  (namely, $ \sup_\Sigma  u_n  \to +\infty$) of solutions to 
\begin{equation}\label{ecua-compact}
 \begin{cases}
 - \Delta u_n + 2 \tilde{K}_n(x) = 2 K_n(x) e^{u_n}, & \hbox{ in } \Sigma, \\
 \frac{\partial u_n}{\partial n} + 2 \tilde{h}_n(x) = 2 h_n(x) e^{u_n/2}, & \hbox{ on } \partial \Sigma, 
 \end{cases}
\end{equation}
with $\tilde{K}_n \to \tilde{K}$, $K_n \to K$, $\tilde{h}_n \to \tilde{h}$ and $h_n \to h$ in the $C^1$ sense, with $K<0$. Define the singular set 
$S$ as in  \eqref{singularset}, and 

\begin{equation}
\label{chi_n} \chi_n=\int_{\S} \tilde{K}_n + \oint_{\partial \S} \tilde{h}_n.
\end{equation}

Then the following assertions hold true:

\begin{enumerate}
\item $S \subset \{ p \in \partial \Sigma: \ \mathfrak{D}(p) \geq 1\}$. 

\item If $\int_{\Sigma} e^{u_n}$ is bounded, then there exists $m \in \N$ such that 
$$S= \{p_1, \dots p_m \} \subset  \{p\in \partial \S : \mathfrak{D}(p) > 1 \}.$$

In this case $|K_n|e^{u_n} \weakto \sum_{i=1}^m \b_i \delta_{p_i}$, $h_ne^{u_n/2} \weakto \sum_{i=1}^m (\b_i+2\pi) \delta_{p_i}$ for some suitable $\b_i > 0$. 
In particular,
$\ \chi_n \to 2 \pi m$. 

\item If $\int_{\Sigma} e^{u_n}$ is unbounded,  there exists a unit positive measure $\sigma$ on $\Sigma$ such that:
\begin{itemize}

\item[a)] $ \displaystyle \frac{|K_n|e^{u_n}}{\int_\Sigma |K_n|e^{u_n}} \weakto \sigma$, \quad $ \displaystyle \frac{h_ne^{u_n/2}}{\oint_{\partial \Sigma} h_n e^{u_n/2}} \weakto \sigma|_{\partial \Sigma}$; 

\item[b)] $ supp \ \sigma \subset \{p  \in \partial \Sigma:\ \mathfrak{D}(p)\geq 1,\ \mathfrak{D}_{\tau}(p) =0\}.$

\end{itemize}

\medskip
\item If there exists $m \in \N$ such that $ind(u_n) \leq m$ for all $n$, then $S= S_0 \cup S_1$, where:
$$S_0 \subset  \{p \in \partial \Sigma:\ \mathfrak{D}(p) = 1,\ \mathfrak{D}_{\tau}(p) =0\},$$
$$S_1 =\{p_1,\dots p_k\}  \subset  \{p\in\partial\S : \mathfrak{D}(p) > 1 \}, \ \ k \leq m.$$
If moreover $\chi_n \leq 0$, then $S_1$ is empty.

\end{enumerate}
\end{theorem}

\

In the above statement, $\mathfrak{D}_\tau(p)$ refers to the derivative of $\mathfrak{D}$ with respect to the tangential direction $\tau$ 
along $\partial \Sigma$. Moreover, $ind(u_n)$ stands for the Morse index of the function $u_n$, namely the maximal dimension of a subspace $E \subset H^1(\Sigma)$ such that the quadratic form $ Q_n: H^1(\Sigma) \to \R$
\begin{equation} \label{quadratic} Q_n(\psi)=  \int_{\Sigma} [ |\nabla \psi|^2 + 2 |K_n(x)| e^{u_n} \psi^2] -  \oint_{\partial \Sigma} h_n e^{u_n/2}\psi^2, \end{equation}
is negative definite when restricted to $E$. 


We will also give some description on the asymptotic profiles of $u_n$ around the singular points, see Proposition \ref{p:6.3} {\em a)} and {\em b)}.

Some comments are in order:

\begin{enumerate}

\item In general we cannot ensure that the singular set $S$ is finite. Indeed we will show explicit examples, given by \cite{Asun}, in which it is not, see Subsection 2.2. This is an entirely new feature of this kind of blow-up. 



\item We are able to pass to a limit problem around any singular point $p \in S$. This is tricky since $p$ need not be isolated and we cannot guarantee the existence of local maxima around such a point. We do this by choosing carefully a sequence $x_n \to p$ with $u_n(x_n) \to +\infty$, even if $u_n(x_n)$ are not local maxima. Our main tool for that is the {\em Ekeland variational principle} (in a finite-dimensional fashion). 

We point out that this is a technical novelty even for the finite mass case and allows one to pass to a limit problem, as in \cite{li-sha}, without knowing a priori that $S$ is finite. \medskip

\item As we can see, the terms  $\int_{\Sigma} |K_n| e^{u_n}$, $\oint_{\partial \S} h_n e^{u_n/2}$ can become unbounded, but when normalized they converge to the same measure. Hence they are in strong competition, and this cancellation implies that a normalization technique is not of use here. Despite, we are able to use a Pohozaev-type identity to show that $\supp \sigma \subset \{ p \in \partial \S: \ \mathfrak{D}_\tau(p)=0 \}$. The main obstacle here is that we do not have any control on the Dirichlet energy of the solutions. We are able to bypass this problem by using {\em holomorphic maps} as test fields in the Pohozaev-type identity, see Proposition \ref{p:int-parts} for more details.

\medskip

\item The limit problem here is posed in a half-plane:

\begin{equation}\label{eclimit}
\left\{\begin{array}{ll}
\displaystyle{-\Delta v = - 2  e^v},  \qquad & \text{in $\R^2_+$,}\\
\ \displaystyle{\frac{\partial v}{\partial n} = 2 \mathfrak{D}(p) e^{v/2}}, \qquad  &\text{on
$\partial \R^2_+ $.}
\end{array}\right.
\end{equation}
In the spirit of \cite{chen-li}, solutions of the above problem with finite mass have been classified in \cite{zhang}, and they exist only if $\mathfrak{D}(p)>1.$

 However we shall need a classification of \emph{all solutions}, with bounded or unbounded mass, that has been given in \cite{mira-galvez} (in the spirit of Liouville, \cite{Liouville}). Those solutions exist only for $\mathfrak{D}(p) \geq 1$ and are linked to the family of holomorphic maps from the half-plane to disks on the hyperbolic space, so they form a much broader family of solutions. However 
in Section 4 we are able to completely characterize the stability of all such solutions, and this allows us
to exclude many blow-up profiles by means of the Morse index restriction. As a consequence the unique possible limit profile with infinite mass is the solution of \eqref{eclimit} with $\mathfrak{D}(p)=1$, namely the $1$-D function:

$$
v(s,t)= 2 \log \left( \frac{\lambda}{1 + \lambda t} \right), \quad \lambda >0.
$$

Observe that there could also be finite-mass blow-up on a finite number of points, if $\chi_n \geq 0$. It is an interesting open question whether those two types of blowing-up behavior can coexist in general.

\item Assumptions on the boundedness of the Morse index of solutions are relatively natural in the framework of minimal surfaces (we 
refer the reader e.g. to the papers \cite{fc}, \cite{ss}, \cite{sharp}, and also to \cite{farina} in a PDE setting), but are completely new in this kind of problems. 
Moreover, we can show an explicit family of blowing-up solutions in the annulus, given in \cite{Asun}, with diverging Morse index and with a different limit profile, see Subsection 2.2.

\end{enumerate}

\medskip The rest of the paper is organized as follows. In Section 2 we discuss the naturality of our assumptions, showing that they 
are in some cases sharp. Section 3 is devoted to the variational study of the functional $I$, completing the proof of Theorem \ref{teo1} and \ref{teo2}. We also combine the  monotonicity trick of Struwe (in a version by Jeanjean) with index bounds by Ghoussoub and Fang. In this way we find solutions of perturbed problems with Morse index 1, under the setting of Theorem \ref{teo3}. 
In Section 4 we recall the classification results for entire solutions in half spaces and determine their Morse index. 
The blow-up analysis of this kind of sequences is started in Section 5. Section 6 is devoted to the unbounded mass case, whereas the bounded mass or bounded Morse index case is studied in Section 7. We conclude the paper with an Appendix devoted to evaluate certain test functions on the energy functional $I$. 
 
\medskip {\bf Acknowledgements:} D.R. wishes to thank J.A. G\'{a}lvez for several discussions on his classification result given in \cite{mira-galvez}, which has been of great help in the elaboration of Section 4.

\medskip

{\bf Notation.} An open ball of center $p$ and radius $r$ will be denoted by $B_p(r)$, while $A(p; r, R)$ stands for an annulus of center $p$ and radii $0 < r < R$. We also use the notation:

$$ B_p^+(r)= \{(s,t )\in B_p(r) \subset \R^2:\ t \geq 0\};$$ $$\Gamma_p^+(r)=\{(s,t) \in \partial B_p^+(r): t=0 \}, \ \ \partial^+ B_p(r) = \partial B_p^+(r) \setminus  \Gamma_p^+(r).
$$

We shall use the symbols $o(1)$, $O(1)$ in a standard way to denote quantities that converge to $0$ or are bounded, respectively. Analogously, we will write $o(\rho)$, $O(\rho)$ to denote quantities that, divided by $\rho$, converge to $0$ or are bounded, respectively.

Given any function $f$, we denote by $f^-= \min\{f,0\}$  and $f^+= \max\{f,0\}$, so that $f= f^+ + f^-$. Moreover, many times in the paper we drop the element of area or length, that is we shall only write $\int_\S K e^u$ or $\int_{\partial \S} h e^{u/2}$; these expressions refer to the standard measure given by the metric. We also use the symbol $\fint f$ to denote the mean value of $f$, that is $$\fint_{\Sigma} f = \frac{1}{\vert \Sigma \vert}\int_\S f.$$

\

\section{Obstructions to existence and examples of blow-up} \label{s:2}
\setcounter{equation}{0}

In this section we include two types of comments related to our problem. First we show some obstructions for the existence of solutions to problem \eqref{ecua}, which somehow complement the existence results given in Theorems \ref{teo1}, \ref{teo2}, \ref{teo3}. Secondly, we will provide some explicit examples of blowing-up solutions for problem \eqref{ecua-compact} with diverging mass, in order to highlight the differences of this phenomenon with respect to the more standard finite-mass blow-up.

\subsection{Obstructions to existence}

In \cite{Asun} the problem of prescribing constant curvatures $K$, $h_1$, $h_2$ on an annulus is treated. In this case all solutions are known explicitly; in particular, the following result is proved (see \cite{Asun}, [Corollary 1, Section 3]):

\begin{theorem} \label{asun} Let $\S$ be an annulus with $K\equiv-1$ and constant geodesic curvatures  $h_1$, $h_2$ at the two boundary components. Then:

\begin{enumerate}
\item[i)]  either $h_1 +h_2>0$ and both $h_i<1$,

\item[ii)]  or $h_1 + h_2<0$ and some $h_i>1$, 

\item[iii)] or  $h_1=1$, $h_2=-1$ or vice-versa. 

\end{enumerate}

\end{theorem}

As it can be seen, Theorem \ref{teo2} is the counterpart of \emph{i)}, whereas Theorem \ref{teo3} is related to  \emph{ii)}. Case  \emph{iii)} is somehow borderline.

\bigskip The assumptions of Theorems \ref{teo1}, \ref{teo2}, \ref{teo3} are natural also in view of the following result:

\begin{theorem} \label{rosenberg}  Let $\S$ be a compact surface with boundary $\partial \S$, and denote by $h$ its geodesic curvature. Assume that $h(x) > \sqrt{|K^-(p)|}$ for all $x \in \partial \S$, $p \in \S$. Then $\Sigma$ is homeomorphic to a disc.

\end{theorem} 

\begin{proof} The proof follows closely an argument by Rosenberg in \cite{Ros} (see also \cite[Appendix A]{Sun}), hence we will be sketchy. Assume by contradiction that $\Sigma$ is not a topological disk, and consider two cases: 
	
\medskip

{\bf Case 1:}  $\partial \S$ is not connected.	
	
Let $\Lambda_0$ be a connected component of $\partial \S$, and $\gamma$ be a curve minimizing the distance between $\Lambda_0$ and $\partial \S \setminus{\Lambda_0}$. Namely, $\gamma$ is a solution of the problem:

$$ Inf \{ Length(\gamma):\ \gamma:[0, L] \to \S,\ \gamma(0) \in \Lambda_0, \ \gamma(L) \in \partial \S \setminus{\Lambda_0}\}.$$

We can assume that $\gamma$ is parametrized by arc-length, and set $p = \gamma(0)$, $q=\gamma(L)$. Take a neighbourhood $O \subset \S$  of  $\gamma[0,L]$ and define $\Lambda_{l} \subset  O$ to be a parallel curve to $\Lambda_0$ at distance $l \in [0,L]$, that is:

$$ \Lambda_l = \{ x \in O: \ d(x, \Lambda_0)= l\},$$
where $d$ denotes the geodesic distance in $\S$. Using the above curvature condition,  it was shown in \cite[Appendix A]{Sun} that $\Lambda_l$ is a regular curve if $O$ is a sufficiently small neighbourhood. Define $\hat{h}(t)$ to be the geodesic curvature of the curve $\Lambda_l$ at the point $\gamma(t)$. It can be shown (see again \cite[Appendix A]{Sun}) that $\hat{h}(t)$ is an increasing function of $t$, so $\hat{h}(L) > \hat{h}(0)= h(p)$. Observe now that $\Lambda_l$ is in contact with $\partial \S \setminus \Lambda_0$, so  $h(q) \leq - \hat{h}(L)$, a contradiction.

\medskip

{\bf Case 2:} $\partial \S$ is connected.

Let us denote by $\tilde{\S}$ the compact surface obtained by identifying $\partial \S$ to a single point. If $\S$ is not a topological disk, then $\tilde{\S}$ is not homeomorphic to a sphere, hence it is not simply connected. Consider the following minimization problem:

$$ Inf \left \{ Length(\gamma):\ \gamma:[0, L] \to \S,\ \{ \gamma(0),\ \gamma(L)\} \subset \partial \S, \ [\tilde{\gamma}] \mbox{ not trivial} \right \}.$$

\begin{figure}[h]
\centering
\includegraphics[width=0.4\linewidth]{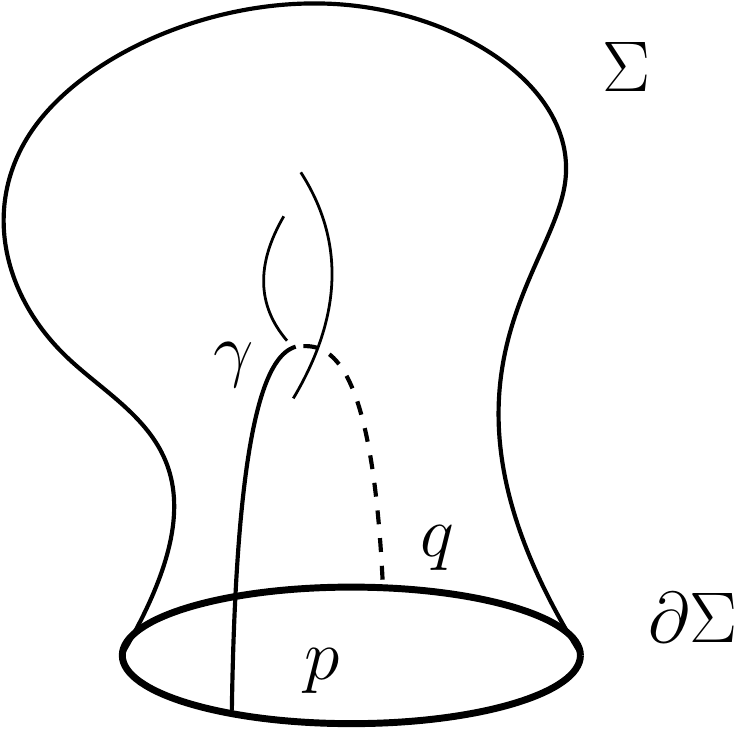}
\caption{A curve minimizing the distance in $\Sigma$ with non-trivial homotopy class in $\tilde{\S}$}
\label{fig:picture}

\end{figure}

Above, we are denoting $\tilde{\gamma}: [0,L] \to \tilde{\S}$ to be the curve obtained by applying the identification to $\gamma$, and $[\tilde{\gamma}]$ to be its homotopy class. Clearly this infimum is attained at a certain $\gamma$ (as in Figure~1), which is a local minimizer of the distance. We can now follow the arguments of Case 1 and still arrive to a contradiction.

\end{proof}

Let us point out that in  case of a disk we can have $h(x) > \sqrt{|K^-(p)|}$ for all $x \in \partial \S$, $p \in \S$. For instance, in the Poincar\'{e} disk there are balls with constant geodesic curvature $h>1$. 

\subsection{Explicit examples of infinite mass blow-up} Consider the problem:

\begin{equation}\label{ecua-asun}
\left\{\begin{array}{ll}
\displaystyle{-\Delta u = - 2 e^u},  \qquad & \text{in $A(0;r,1)$,}\\
\displaystyle{\frac{\partial u}{\partial n} + 2= 2h_1 e^{u/2}}, \qquad  &\text{on
$ |x|=1$,} \\
\displaystyle{\frac{\partial u}{\partial n} - 2/r = 2h_2 e^{u/2}}, \qquad  &\text{on
$ |x|=r$,}
\end{array}\right.
\end{equation}
where $h_1$, $h_2$ are two real constants. All explicit solutions of this problem have been classified by \cite{Asun}; here we exhibit some of them observing that they form families of blowing-up solutions with infinite mass.

As a first example (see \cite[Equation (3) and Lemma 2]{Asun}), the functions

$$ u(x) = \log \left( \frac{4}{|x|^2(\lambda + 2 \log |x|)^2} \right), \quad \mbox{ for any } \lambda \notin [0, -2 \log r], $$
are solutions of \eqref{ecua-asun}; if $\lambda<0$ then $h_1=1$ and $h_2=-1$, whereas if $\lambda> -2 \log r$, $h_1=-1$ and $h_2=1$.  Observe that if $\lambda$ tends to $0$ or $2 \pi r$, then the functions $u$ blow up at a whole component of the boundary, which is indeed the one with curvature equal to $1$. In particular, the blowing-up set $S$ is infinite here.

\bigskip Let us give a second example. Given any $h_1>1$, for any $\gamma \in \N$, there exists an explicit solution in the form (we use complex variable notation):

\begin{equation} \label{uasun} u_\gamma(z)= 2 \log \left ( \frac{ \gamma |z|^{\gamma-1}}{h_1+Re(z^{\gamma})}   \right ), \end{equation}
where $h_2 =-h_1 r^{-\gamma}$. See \cite[equation (5) and Lemma 2]{Asun}.  

Observe that the  solutions in \eqref{uasun} form a blowing-up sequence as $\gamma \to +\infty$, keeping $h_1>1$ fixed. Indeed,

$$ u_\gamma(z) \to - \infty \ \mbox{ if } |z|<1,\quad u_\gamma(z) \to +\infty  \ \mbox{ if } |z|=1.$$

Hence in this case $S= \{|z|=1\}$ and $h_1>1$; it is also easy to check that  $\sigma = \delta_{\{|z|=1\}}$. Let us compute the asymptotic profile of these solutions. Fix a point of blow-up $z=-1$, and define the rescaling:

$$ v_\gamma(z)= u_\gamma \left (1-\frac{z}{\gamma} \right ) - 2\log \gamma.$$

Hence $v$ can be written in the form:

$$ v_\gamma(z)= 2 \log \left ( \frac{ |1-\frac{z}{\gamma}|^{\gamma-1}}{h_1+Re((1-\frac{z}{\gamma})^{\gamma})}   \right ).$$

Clearly,

$$ \lim_{\gamma \to +\infty} \left (1-\frac{z}{\gamma} \right )^\gamma = e^{-z},$$
so  $v_\gamma$ converges, at least point-wise, to the function:

$$v(z)= 2 \log \left ( \frac{ |e^{-z}|}{h_1+Re(e^{-z})}   \right )= 2 \log \left ( \frac{ e^{-t}}{h_1+e^{-t} \cos s}   \right ),$$
which is defined in the half-plane $\{(s,t) \in \R^2:\ t \geq 0\}$. This is indeed a solution to the limit problem in the half-space with $K=-1 $, $h_1>1$, infinite mass and infinite Morse index, according to our results in Section 4. 

This example shows that infinite Morse index blowing up solutions exist, even in this geometric context. Moreover, the blow-up limits can be entire solutions in the half-plane with infinite Morse index. Let us point out, though, that this example is not under the conditions of Theorem \ref{teo4} since $h_2 = -h_1 r^{-\gamma}$ diverges negatively. Whether the infinite Morse index assumption is necessary or not in Theorem \ref{teo4} {\em (4)} remains as an interesting open problem.

\

\section{Study of the energy functional}
\setcounter{equation}{0}

In this section we will perform a variational study of the energy functional $I$ defined in \eqref{functional}. From this we will derive the proof of Theorems \ref{teo1}, \ref{teo2}, and also \ref{teo3}, provided that Theorem \ref{teo4} holds. This last result, which addresses the blow-up analysis, will be proved later.

As a first preliminary result, let us show that we can always prescribe zero geodesic curvature and constant Gaussian curvature. 

\begin{proposition} \label{easy} Equation \eqref{ecua0} is always solvable if $h=0$ and
$$ K = 0 \mbox{ if } \chi(\S)=0,\ \ \ \ K=-1  \mbox{ if } \chi (\S) <0, \ \ \ \ K=1  \mbox{ if } \chi(\S)=1. $$
The solution is unique up to a constant in the first  case, unique in the second case, and 
 unique up to  M\"obius transformations in the third case.

\end{proposition}

\begin{proof}

If $\chi(\Sigma)<0$,  we need to study the (strictly convex) energy functional:

$$J(u)= \int_{\S} \left(  \frac 1 2 |\nabla u|^2 + 2 \tilde{K}u + 2 e^u \right) + 2 \oint_{\partial \S} \tilde{h} u.$$

Taking into account that $\int_{\S} \tilde{K} + \oint_{\partial \S} \tilde{h} = 2 \pi \chi(\S) <0$, it is easy to show the existence of a unique minimizer for $J$.

If $\chi(\Sigma)=0$, we just need to find a solution to:
$$
\left\{\begin{array}{ll}
\displaystyle -\Delta u +2 \tilde{K}= 0,  \qquad & \text{in $\Sigma$,}\\
\displaystyle \frac{\partial u}{\partial n} + 2 \tilde{h} = 0, \qquad  &\text{on
$\partial\Sigma $.}
\end{array}\right.
$$

This problem is linear and has a unique solution, up to addition of constants.

Finally, if $\chi(\S)=1$ then $\S$ is topologically a disk. By the Uniformization Theorem, $\S$ is conformally equivalent to the hemisphere 
so the statement follows. 
\end{proof}

With this result at hand, we can  assume that our initial metric has constant Gaussian curvature and that $\partial \S$ is geodesic. Hence our problem reduces to \eqref{ecua}. As commented in the introduction, \eqref{ecua} is the Euler-Lagrange equation of the energy functional $I$ defined in \eqref{functional}. The following estimate will be crucial in the variational study of $I$.

\begin{lemma}\label{trace}
Define $\bar{\mathfrak{D}}= \max\{ \mathfrak{D}^+(x), \ x \in \partial \S\}$. For any $\e >0$ there exists $C>0$ such that:

$$ 4  \oint_{\partial \S} h e^{u/2} \leq ( \bar{\mathfrak{D}}  + \e)  \left [  \int_{\Sigma} \frac 1 2 |\nabla u|^2  + 2 |K| e^u \right ]+C.$$
\end{lemma}

\begin{proof}
Take a finite partition of  unity $\{ \phi_j\}_{j=1}^M$ of $\partial \Sigma$ such that for all $j$ one has $diam (supp \, \phi_j) < \delta$ for some fixed $\delta >0$.  We denote by $\S_j$ the support of $\phi_j$. Take a smooth vector field $N$ in $\S$ such that $N(x)=  \nu(x)$ on the boundary, $|N(x)|\leq 1$. Then, we use the divergence theorem to obtain:

\begin{align*}  \oint_{\partial \S} \phi_j e^{u/2} = \oint_{\partial \S} \phi_j  e^{u/2}  N(x) \cdot \nu(x)  \\ = \  \int_\S e^{u/2} \left [ \nabla \phi_j \cdot N  + \phi_j  div\, N  + \frac 1 2 \phi_j \nabla u \cdot N \right ]   \leq \  C \int_\S  e^{u/2} + \frac 1 2  \int_\S   e^{u/2}   |\nabla u| \phi_j.  \end{align*}

Let us set $\bar{h}_j= \sup_{x \in \S_j} \{h^+(x)\}$. If the diameter of $\S_j$ is small enough, one has that $\bar{h}_j \leq ( \bar{\mathfrak{D}}  + \e) \sqrt{|K(x)|}$ for any $x \in \S_j$. Then, we apply Schwartz's inequality:

\begin{align*} 4 \oint_{\partial \S} h e^{u/2} = 4 \sum_{j=1}^M \oint_{\partial \S}  \phi_j h e^{u/2} \leq 4 \sum_{j=1}^M \bar{h}_j \oint_{\partial \S}  \phi_j  e^{u/2} \\  \leq    C \int_\S  e^{u/2} + ( \bar{\mathfrak{D}}  + \e)    \sum_{j=1}^M 2  \int_\S \sqrt{|K|} \phi_j e^{u/2}|\nabla u|   \\ \leq  C \int_\S  e^{u/2} + ( \bar{\mathfrak{D}}  + \e) \left (   \sum_{j=1}^M 2  \int_\S |K| \phi_j e^{u}   + \frac 1 2 \int_{\S} \phi_j |\nabla u|^2  \right ) \\ = C \int_\S  e^{u/2} + ( \bar{\mathfrak{D}}  + \e)   \left ( 2  \int_\S |K| e^{u}   + \frac 1 2 \int_{\S}  |\nabla u|^2 \right ). \end{align*}

We conclude by observing that $\int_\S e^{u/2} \leq \e \int_\S e^{u} + C$ and renaming $\e$ appropriately.

\end{proof}

\begin{proof}[Proof of Theorem \ref{teo1}]

Under the assumptions of Theorem \ref{teo1}, recalling the notation from Lemma \ref{trace} we have $\bar{\mathfrak{D}}<1$, so  we can choose $\e>0 $ with $\bar{\mathfrak{D}} + \e < 1-\e$. By inserting the inequality in Lemma \ref{trace} into the definition of $I$, we obtain:

$$I(u) \geq    \int_{\Sigma} \left( \frac{\e}{2} |\nabla u|^2  + 2 \e  |K| e^u  + 2  \tilde{K} u \right) - C.$$

Since $\tilde{K}<0$,  $$\lim_{u \to \pm \infty}  2 \e  |K| e^u  + 2  \tilde{K} u = +\infty.$$ 

Then $I$ is coercive. It is also standard to check that $I$ is weakly lower semicontinuous, so that it attains its infimum.

\medskip If moreover $h \leq 0$, then it is easy to check that the functional $I$ is strictly convex. As a consequence its minimum corresponds to its unique critical point.

\end{proof}

\begin{proof}[Proof of Theorem \ref{teo2}]
Here $\tilde{K}=0$ and by Lemma \ref{trace} we just obtain:
\begin{equation} \label{piu} I(u) \geq    \int_{\Sigma} \left( \frac{\e}{2} |\nabla u|^2  + 2 \e  |K| e^u \right) - C.\end{equation}

This implies that $I$ is bounded from below but we do not have coercivity in this case. First, let us show that $\inf I <0$. Indeed, take $u=-c$, with $c>0$. Then:

$$I(-c)= e^{-c} \int_{\Sigma} 2 |K| - 4 e^{-c/2} \int_{\partial \Sigma} h,$$
which is negative for large $c$, since $\int_{\partial \Sigma} h >0$.

Take now  a minimizing sequence $u_n$ for $I$. By \eqref{piu}, $\int_{\Sigma} |\nabla u_n|^2 $ is bounded. We conclude if we show that $\fint_{\Sigma} u_n$ is bounded. 

Up to a subsequence, we can assume that 

$$ u_n - \fint_{\Sigma} u_n \weakto u_0 \quad \mbox{ in } H^1(\Sigma) . $$ As a consequence,$$e^{ u_n - \fint_{\Sigma} u_n} \to e^{u_0} \mbox{ in } L^1(\Sigma), \   e^{ (u_n - \fint_{\Sigma} u_n)/2} \to e^{u_0/2} \mbox{ in } L^1(\partial \Sigma).$$

If now $\fint_{\Sigma} u_n \to + \infty$, then

$$ \int_{\S}  |K| e^{u_n} = e^{\fint_{\Sigma} u_n}  \int_{\S}  |K| e^{u_n- \fint_{\Sigma} u_n} \to +\infty,$$
contradicting \eqref{piu}.

So, we now consider the remaining case $\fint_{\Sigma} u_n \to - \infty$. In such a case,

$$ \oint_{\partial \S}  h \, e^{u_n/2} = e^{\frac{\fint_{\Sigma} u_n}{ 2}}  \oint_{\partial \S} h \, e^{(u_n- \fint_{\Sigma} u_n)/2} \to 0.$$

By the definition of $I$ we conclude that

$$ \liminf_{n\to +\infty} I(u_n) \geq 0,$$ which is a contradiction.

\end{proof}

In what follows we show that under the assumptions of Theorem \ref{teo3}, the functional $I$ has a min-max structure. 

\begin{proposition} \label{prop-minmax} Assume that $\tilde{K}=0$, $K < 0$, $\oint_{\partial \S} h <0$ and $\mathfrak{D}(p) >1$ for some $p \in \partial \S$. Then there exist $u_0$, $u_1 \in H^1(\S)$ such that:

$$ c = \inf_\gamma  \max_t \{ I(\gamma(t)):\ t \in [0,1],\ \gamma \in \Gamma\} > \max \{ I (u_0), I(u_1)\}>0, $$

where $\Gamma = \{ \gamma: [0,1] \to H^1(\S) \mbox{ continuous}: \gamma(0)= u_0, \ \gamma(1)=u_1\}$.

\end{proposition}

The proof of Proposition \ref{prop-minmax} requires two preliminary results:

\begin{lemma} \label{lemma:notbounded} If $K<0$ and $\mathfrak{D}(p)>1$ for some $p \in \partial \S$, then there exists a sequence $u_n$ such that $I(u_n) \to - \infty $ and $\oint_{\partial \S} e^{u_n/2} \to + \infty$. \end{lemma}

\begin{proof}
The proof of this result is postponed to the Appendix. \end{proof}

\begin{lemma} \label{mp-geometry} If $\tilde{K}=0$ and $\oint_{\partial \S} h<0$, then there exists $\e>0$, $\delta>0$ such that 

$$ I(u) > \e \quad \forall \, u \in H^1(\S) \quad \mbox{ with } \quad \oint_{\partial \S} e^{u/2} = \delta.$$

\end{lemma}

\begin{proof}

First we claim that for any $\delta>0$, the infimum:

$$ \alpha_\delta= \inf \left\{ I(u): u \in H^1(\S) \mbox{ with } \oint_{\partial \S} e^{u/2} = \delta\right\}$$
is attained. Recall the definition of $I$ given in \eqref{functional}; this implies that $\alpha_\delta \neq -\infty$. Moreover, for a minimizing sequence $u_n$ the integral $\int_\Sigma |\nabla u_n|^2 $ must be bounded, otherwise $I(u_n) \to +\infty$. Up to a subsequence,

$$ u_n - \fint_{\Sigma} u_n \weakto u_0 \mbox{ in } H^1(\Sigma), $$ which implies that $$e^{ u_n - \fint_{\Sigma} u_n} \to e^{u_0} \mbox{ in } L^1(\Sigma), \   e^{ (u_n - \fint_{\Sigma} u_n)/2} \to e^{u_0/2} \mbox{ in } L^1(\partial \Sigma).$$

Moreover, 

$$ 0<\delta = \oint_{\partial \S} e^{u/2} = e^{\frac{\fint_{\Sigma} u_n}{ 2}}  \oint_{\partial \S} h e^{(u_n- \fint_{\Sigma} u_n)/2},$$
which implies that $\fint_{\Sigma} u_n $ is bounded. Hence $u_n$ is bounded in $H^1(\Sigma)$ and the conclusion follows from  standard arguments.

\medskip

We now turn our attention to the proof of Lemma \ref{mp-geometry}. By contradiction, assume that there exists $u_n \in H^1(\S)$ such that:

$$ \liminf_n I(u_n) \leq 0, \qquad \qquad  \oint_{\partial \S} e^{u_n/2} \to 0.$$ 

By the definition of $I$ it readily follows that $\int_{\Sigma } |\nabla u_n|^2 \to 0$, hence $u_n - \fint_{\Sigma} u_n \to 0$ in $H^1(\S)$. Therefore,

$$ I(u_n)= \int_{\Sigma} \frac 1 2 |\nabla u_n|^2  + 2 |K| e^{u_n} - 4 \int_{\partial \Sigma} h e^{u_n/2} \geq - 4 e^{\fint_{\Sigma} u_n/2}\int_{\partial \Sigma} h e^{(u_n/-\fint_{\Sigma} u_n)/2} ,$$
but $\int_{\partial \Sigma} h e^{(u_n/-\fint_{\Sigma} u_n)/2}  \to \int h <0$, hence $I(u_n) >0$ for sufficiently large $n$.

\end{proof}

\begin{proof}[Proof of Proposition \ref{prop-minmax}]

Fix $\e>0$, $\delta>0$ as given by Lemma \ref{mp-geometry}. Take a constant $c>0$  so large   that 

$$I(-c) = 2 \int_{\S} |K| e^{-c} - 4 \oint_{\partial \S} h e^{-c/2} \in \left( 0, \frac{\e}{2} \right), $$ and moreover $\oint_{\partial \S} e^{-c/2} < \delta$. We define $u_0=-c$. 

By Lemma \ref{lemma:notbounded}, there exists also $u_1$ with $\oint_{\partial \S} e^{u_1/2} > \delta$ satisfying  $I(u_1) <0$. 

Observe that for any $\gamma \in \Gamma$ there exists $t \in (0,1)$ such that $\oint_{\partial \S} e^{\gamma(t)/2} = \delta$. As a consequence, $c \geq \e > \max \{I(u_0),\ I(u_1)\}>0$.

\end{proof}

The main issue in order to prove Theorem \ref{teo3} is the fact that we do not know whether the Palais-Smale condition holds or not for $I$. The strategy consists in producing a sequence of solutions $u_n$ to a perturbed version of \eqref{ecua}, namely \eqref{ecua-compact} (where $\tilde{K}_n$ and 
$\tilde{h}_n$ are as in Theorem \ref{teo4}), with Morse index not exceeding $1$.  The main tool is a monotonicity argument originally attributed 
to Struwe (see \cite{str-mon}), combined with a deformation argument from 
\cite{FG}.

For $\e$ close to zero, we consider the following family of functionals 
$$
  I_\e(u) = I(u) + \e J(u), 
$$
where $I$ is as in \eqref{functional}, and 
$$
  J(u) = \int_{\Sigma} \left( |\n u|^2 + e^u - u \right). 
$$
Notice that, since 
$$
   e^u - u \to + \infty \qquad \hbox{ as } u \to \pm \infty, 
$$
$J(u)$ is coercive on $H^1(\Sigma)$. 

For a small $\e_0 > 0$ and $\e \in (0, \e_0 )$ we can reason as in the 
proof of Proposition \ref{prop-minmax}, to find 
 two elements $u_0, u_1 \in H^1(\Sigma)$ for 
which 
\begin{equation}\label{eq:min-max-delta}
  c_\e := \inf_{\g \in \G} \max_{t \in [0,1]} I_\e(\g(t)) > \max \{ I_\e(u_0), I_\e (u_1) \} + \d > \d,
\end{equation}
where $\d > 0$ is a fixed positive number and where $\G$, again, stands for the class of admissible curves 
$$
  \G = \left\{ \g \in C([0,1];H^1(\S)) \; : \; \g(0) = u_0, \g (1) = u_1  \right\}. 
$$
Under these assumptions, the function $\e \mapsto c_\e$ is monotone non-decreasing 
and  therefore $c_\e$ is a.e.  differentiable in $\e$.

Consider a value $\e$ where $c_\e$ is 
differentiable, and let $\e_n \nearrow \e$. 
It is shown in the proof of Proposition 2.1 in \cite{jj} that at such a value, 
if $\gamma_n \in \Gamma$ satisfies 
$$
  \max_{t \in [0,1]} I_{\e} (\gamma_n(t)) \leq c_{\l} + (2 + c'_\l) (\l - \l_n), 
$$
and if $t$ is such that 
$$
  I_\e(\gamma_n(t)) \geq c_\e - (\e - \e_n), 
$$
then $\|\gamma_n(t)\| \leq M$, with $M$ depending on $c'_\e$.

Given $\a > 0$, define 
\begin{equation}\label{eq:Falpha}
  F_\a = \left\{ u \in H^1(\Sigma) \; : \; \|u\| \leq M + 1 \hbox{ and } |I_\e - c_\e| \leq \a  \right\}. 
\end{equation}
The classical deformation lemma is used in \cite{jj} to prove that $I_\e$ has a critical point in $F_\a$. 

On the other hand, in \cite{FG} the following result is shown. Assume that a functional $\mathcal{I}$ on a Hilbert space 
$H$ is of class $C^2$ has a mountain-pass structure with mountain-pass level $c$, and assume there 
exists $\rho > 0$ such that $\mathcal{I}'$ and $\mathcal{I}''$ are uniformly H\"older continuous on the set 
$\{c - \rho \leq \mathcal{I} \leq c + \rho\}$. Then a Palais-Smale sequence $x_n$ at level $c$ is found, satisfying also the 
following second-order property (see Corollary 1 in \cite{FG}).

\medskip {\bf (i)} \quad If $\mathcal{I}''(x_n)[u,u] < - \frac{1}{n} \|u\|^2$ for all $u$ belonging to a subspace $E \subseteq H$, 
then dim$(E) \leq 1.$

\medskip

 Such a sequence is found starting from 
curves $f(t)$, $t \in [0,1]$, for which $\sup_{t \in [0,1]} \mathcal{I}(f(t))$ is close to $c$, and then deforming 
them properly at points for which $\mathcal{I}(f(t))$ is large enough. Such a deformation displaces the path $f(t)$ 
near its highest level only by a small amount, see formula (21) in the proof of Theorem 1.bis in 
\cite{FG}. 

In our case, level sub-supersets of the form $\{c_\e - \rho \leq I_\e \leq c_\e + \rho\}$ are unbounded, and 
therefore the above uniform H\"older continuity property is not guaranteed. However, by the 
above monotonicity argument, the deformation in \cite{FG} can be localized to the  
set $F_\a$ in \eqref{eq:Falpha}, where such uniform H\"older continuity holds true. As a consequence, we obtain 
a bounded Palais-Smale sequence for $I_\e$ at level $c_\e$ also satisfying the property {\bf (i)}. 
Recall that the maps $u \mapsto e^{u}$ and $u \mapsto e^{u/2}|_{\partial \S}$ are compact 
from $H^1(\S)$ into $L^1(\S)$ and $L^1(\partial \S)$ respectively. Therefore, passing to a 
proper subsequence, the limit of the latter Palais-Smale sequence exists and is a critical 
point of $I_\e$ with Morse index less or equal to $1$.

We can summarize the above discussion in the next proposition. 

\

\begin{proposition}\label{p:monot}
 Suppose that we are under the assumptions of Theorem \ref{teo3}. Then there exists a sequence  $\e_n\searrow 0$ and solutions $u_n$ to problem \eqref{ecua-compact}, with 
 $$
   \tilde{K}_n = \frac{- \e_n /2 }{1+2 \e_n}; \quad K_n(x) = \frac{K(x) - \e_n/2}{1+2\e_n}; 
   \quad h_n(x) = \frac{h(x)}{1+2\e_n}; \quad \tilde{h}_n=0.
 $$
Moreover $ind(u_n) \leq 1$. 
 \end{proposition}

\


\medskip

\begin{pfn}{\em \, of Theorem \ref{teo3}}
We consider the sequence $u_n$ given by Proposition \ref{p:monot}, with $\e_n \searrow 0$. Observe in particular that $\chi_n \leq 0$. By Theorem \ref{teo4} {\em (4)} and our assumptions on the function 
$\mathfrak{D}$, which has no critical points in $\{\mathfrak{D} = 1\}$, $u_n$ must be uniformly bounded from above. 

By these upper bounds on $u_n$, the right-hand sides of \eqref{ecua-compact} are uniformly bounded. 
By standard elliptic regularity theory, a subsequence of $u_n - \fint_\S u_n$ converges 
in $C^1(\Sigma)$ to some function $w$.

If we assume by contradiction that $\inf_\Sigma u_n \to - \infty$, then  $\fint_\S u_n \to -\infty$. However this would imply 

$$ e^{u_n} = e^{u_n - \fint_\S u_n} e^{ \fint_\S u_n} \to 0,$$
uniformly on $\S$. As a consequence, $w$ is harmonic in $\Sigma$ and satisfies homogeneous
Neumann boundary conditions. Since $\int_{\Sigma} w = 0$, we must have $w \equiv 0$ 
and therefore, recalling the notation from the beginning of this section, $I_{\e_n}(u_n) \to 0$. 
This is however in contradiction to the fact that $u_n$ is a min-max sequence, with min-max 
value satisfying \eqref{eq:min-max-delta}. 
%
%
%
%
%
%
%
%
\end{pfn}

\

\section{Solutions of the limit problem in  the half-space and their Morse index}\label{app-A}\setcounter{equation}{0}

In this section we study the Morse index of  solutions to the limit problem:

\begin{equation}\label{equaplane0}
\left\{\begin{array}{ll}
\displaystyle{-\Delta v = 2 K_0 e^v},  \qquad & \text{in $\mathbb{R}^2_+$,}\\
\displaystyle{\frac{\partial v}{\partial n} = 2h_0 e^{v/2}}, \qquad  &\text{on
$\partial \mathbb{R}^2_+$,}
\end{array}\right.
\end{equation}
where $K_0<0$, $h_0$ are constants, and $\R^2_+ =\{ (s,t) \in \R^2:\ t \geq 0\}$. 
The solutions of this problem have been classified in \cite{mira-galvez, zhang}, and they are as follows:

\begin{theorem} Define:  $\mathfrak{D_0} = \frac{h_0}{\sqrt{|K_0|}}.$ The following assertions hold true:

\begin{enumerate}
\item If $\mathfrak{D_0}<1$ then \eqref{equaplane0} does not admit any solution.

\item If $\mathfrak{D_0}=1$ the only solutions of \eqref{equaplane0} are given by:

\begin{equation} \label{profile2} v_\lambda(s,t)= 2 \log \Big ( \frac{\lambda}{1 + \lambda t} \Big ) - \log |K_0|, \ \lambda >0. \end{equation}

\item If $\mathfrak{D_0}>1$ there exists a  locally univalent holomorphic map $g$ from $\mathbb{R}^2_+$ to a disk of geodesic curvature $\mathfrak{D_0}$ in the Poincar\'{e} disk $\mathbb{H}^2$ such that

$$v(z)= 2 \log \left( \frac{2|g'(z)|}{1-|g(z)|^2} \right)  - \log |K_0|.$$

We recall that $\mathbb{H}^2$ is the unit disk in $\mathbb{C}$ equipped with the metric $\frac{4}{(1-|z|^2)^2 } \, dz$. 

\medskip Moreover, $g$ is a M\"{o}bius map if and only if

\begin{equation} \label{finitemass} \mbox{ either } \int_{\R^2_+} e^v < + \infty \ \mbox{ or }  \oint_{\partial \R^2_+} e^{v/2} < +\infty.\end{equation}
In such case $v$ can be written as:

\begin{equation}\label{bubble} v_\lambda(s,t)= 2 \log \left( \frac{2 \lambda}{(s-s_0)^2 + (t+t_0)^2 - \lambda^2}  \right)  - \log |K_0|,
\end{equation}
where $\lambda >0$, $s_0 \in \R$, $t_0 = \mathfrak{D_0} \lambda$. Moreover,

$$ \int_{\R^2} |K_0| e^{v_\lambda} = \beta , \ \ \oint_{\partial \R^2_+} h_0 e^{v_\lambda/2}= \beta + 2 \pi,$$

with 

\begin{equation}\label{eq:beta}
\beta=2\pi \left (\frac{h_0}{\sqrt{h_0^2+K_0}} - 1 \right ).
\end{equation}

\end{enumerate}

\end{theorem}

\bigskip We will next compute the Morse index of the above solutions. By the change of variable $v= v+ \log (|K_0|)$, we can assume $K_0=-1$ and pass to the equivalent problem:

\begin{equation}\label{equaplane}
\left\{\begin{array}{ll}
\displaystyle{-\Delta v = -2 e^v},  \qquad & \text{in $\mathbb{R}^2_+$,}\\
\displaystyle{\frac{\partial v}{\partial n} = 2 \mathfrak{D_0} e^{v/2}}, \qquad  &\text{on
	$\partial \mathbb{R}^2_+$.}
\end{array}\right.
\end{equation}
We will be concerned with the study of the quadratic form:

\begin{equation} \label{quadratic0}
Q(\psi)=  \int_{\mathbb{R}^2_+} |\nabla \psi|^2 dV_g + 2 \int_{\mathbb{R}^2_+} e^{v} \psi^2 \, dV_g  - \mathfrak{D_0} \oint_{\partial \mathbb{R}^2_+} e^{v/2} \psi^2 \, dy_g,
\end{equation}
where $\psi\in C^\infty_0(\mathbb{R}^2_+)$, the set of test functions with compact support (not necessarily zero on the boundary). We define the Morse index of a solution $v$ of \eqref{equaplane} as:

$$ind(v)= \sup \{ dim(E): \ E \subset \ C_0^{\infty}(\R^2_+) \ \mbox{ vector space, } Q(\psi) < 0 \ \forall \psi \in E \}.$$

We understand that $ind(v) = +\infty$ if the above set is not bounded from above.

\begin{theorem}\label{t:app-index}   Let $v$ be a solution of  problem \eqref{equaplane}. Then:

\begin{enumerate}

\item[a)] If $\mathfrak{D_0} =1$, then $ind(v)=0$, namely $v$ is stable.

\item[b)] If $\mathfrak{D_0} >1$, then:

If \eqref{finitemass} is satisfied, then $ind(v)=1$. Otherwise, $ind(v)= +\infty$.

\end{enumerate}

\end{theorem}

\begin{proof}

In case a) the solution $v$ is given by \eqref{profile2}. Let us consider the linearized problem:

$$
\left\{ \begin{array}{ll}
\displaystyle{-\Delta \psi + 2 \frac{1}{(1+t)^2} \psi= 0},  \qquad & \text{in } \R^2_+,\\
\displaystyle{\frac{\partial \psi}{\partial n}} =   \psi , \qquad  &\text{on }
\partial \R^2_+.

\end{array}\right.
$$

An explicit solution is $\psi(s,t)= \psi(t)= \frac{1}{t+1}$, which is a positive function. This implies stability (see for instance \cite{dupaigne}, Section 1.2) .

In case b), let $g$ be given by the above classification. By composing with a symmetry of $\mathbb{H}^2$ we can assume, without loss of generality, that $g(\R^2_+)$ is contained in the disk $D_R \subset \mathbb{H}^2$ centred at $0$. Here $R$ denotes its euclidean radius, that satisfies  $R= \mathfrak{D_0} - \sqrt{\mathfrak{D_0}^2-1}<1$. In such a case, the Gaussian and geodesic curvatures of $D_R$ in $\mathbb{H}^2$ translate to $\mathbb{R}_+^2$ via $g$. 

In what follows we shall write $\rho(s)= \frac{4}{(1-s^2)^2}$. In order to study stability  we pass to the disk $D_R$ and study the following quadratic form:

$$
Q_R(\psi)= \int_{D_R} |\nabla \psi|^2 dz + 2 \int_{D_R} \psi^2 \, \rho(|z|) dz - \mathfrak{D_0}  \oint_{\partial D_R} \psi^2 \sqrt{\rho(R)}  dz,
$$
for $\psi \in C^{\infty}(\overline{D_R})$. We claim that $Q_R$ has Morse index 1. Its associated linear operator can be written as

\begin{equation}\label{equadisc1}
\left\{\begin{array}{ll}
\displaystyle{-\Delta \psi + 2 \rho(|z|) \psi= 0},  \qquad & \text{on $D_R$,}\\
\displaystyle{\frac{\partial \psi}{\partial n}} = \mathfrak{D_0}  \sqrt{\rho(R)} \gamma , \qquad  &\text{on }
\partial D_R.
\end{array}\right.
\end{equation}

One can easily check that the functions:

$$ \frac{x_1}{1-|x|^2},\ \qquad \quad   \frac{x_2}{1-|x|^2},$$ 
satisfy \eqref{equadisc1}. These elements in the kernel are of course related to the invariances of our problem. It is easy to observe, by using Fourier decomposition, that those functions correspond to the second mode expansion. As a consequence, there is a unique negative eigenvalue with radially symmetric eigenfunction. This eigenfunction is indeed explicit:

$$ \gamma(x)= \frac{1+|x|^2}{1-|x|^2}.$$ 

Observe that this function is bounded in $D_R$, for $R \in (0,1)$. Moreover, $\gamma$ solves the  boundary value problem: 

$$
\left\{\begin{array}{ll}
\displaystyle{-\Delta \psi + 2 \rho(|z|) \psi= 0},  \qquad & \text{on $D_R$,}\\
\displaystyle{\frac{\partial \psi}{\partial n}} = \frac{1}{\mathfrak{D_0}}  \sqrt{\rho(R)} \gamma , \qquad  &\text{on }
\partial D_R.
\end{array}\right.
$$

Observe that this equation is very similar to \eqref{equadisc1} , but with $\mathfrak{D_0}$ replaced by $\frac{1}{\mathfrak{D_0}}$. Since $\mathfrak{D_0}>1$, we have that $Q_R(\gamma)<0$. This finishes the proof of the claim.

\bigskip Define now $\psi= \gamma \circ g$. Clearly, $\psi$ solves:

\begin{equation}\label{linear plane}
\left\{ \begin{array}{ll}
\displaystyle{-\Delta \psi + 2 e^v \psi= 0},  \qquad & \text{in } \R^2_+,\\
\displaystyle{\frac{\partial \psi}{\partial n}} = \frac{1}{\mathfrak{D_0}} e^{v/2}  \psi , \qquad  &\text{on }
\partial \R^2_+.

\end{array}\right.
\end{equation}

Let us first consider the case of finite mass, that is, assume that \eqref{finitemass} holds. Using the invertibility of the M\"obius map $g$, we can relate the second variation 
in $D_R$ to that in $\R^2_+$, which implies that $Q(\psi) < 0$. 

We observe now that, by definition, $\psi$ is uniformly bounded in $\R^2_+$; since $e^v \in L^1(\R^2_+)$ and $e^{v/2} \in L^1 (\partial \R^2_+)$ and since 
$\psi$ has a limit at infinity, we can find a compactly-supported function $\tilde{\psi}$ 
such that, still $Q(\tilde{\psi})<0$. As a consequence, $ind(v) \geq 1$. 

If by contradiction we had the strict inequality, we would 
be able to find a two-dimensional space $E_2$ spanned by two (smooth) compactly-supported functions 
such that the above quadratic form would be negative-definite on $E_2$. Considering then 
the two-dimensional space $\hat{E}_2$ of functions on the disk $D_R$ defined by 
$$
  \hat{E}_2 = \left\{ \phi \circ g^{-1} \; : \; \phi \in E_2  \right\},
$$
we would have that $Q_R(\hat{\phi}) < 0$ for all $\hat{\phi} \in \hat{E}_2$. This would 
contradict the fact that the index of $Q_R$ is precisely $1$. Therefore, we proved that 
if \eqref{finitemass} holds then $ind(v) = 1$. 

\medskip

Assume now that $\oint_{\partial \R^2_+} e^{v/2} = +\infty$: we will show that the solution $v$ is not stable outside any compact set. This implies in particular that the Morse index is infinite.  

We multiply \eqref{linear plane} by $\phi^2 \psi$ and integrate, where $\phi$ is a conveniently chosen cut-off function:

$$ 0 = \int_{\R^2_+} (-\Delta \psi + 2 e^v \psi) \phi^2 \psi = \int_{\R^2_+}  \nabla \psi \cdot \nabla (\phi^2 \psi) + 2 e^v \psi^2 \phi^2  - \oint_{\partial \R^2_+} \frac{1}{\mathfrak{D_0}} e^{v/2} \psi^2 \phi^2.$$

Taking into account that $|\nabla (\phi \psi)|^2= \nabla \psi \cdot \nabla (\phi^2\psi) + \psi^2 |\nabla \phi|^2$, we have:

\begin{equation} \label{crucial} \int_{\R^2_+} \psi^2 |\nabla \phi|^2 = \int_{\R^2_+}  |\nabla (\phi \psi)|^2 + 2 e^v (\psi \phi)^2  - \oint_{\partial \R^2_+} \frac{1}{\mathfrak{D_0}}  e^{v/2} (\psi \phi)^2.\end{equation}

Observe that the  above right-hand side is similar to the expression of $Q(\phi \psi)$ given in \eqref{quadratic0}, but with $\frac{1}{\mathfrak{D_0}} $ instead of $\mathfrak{D_0}$.

Indeed, given a fixed $M_0>0$ and $M$ large enough, take $\phi= \phi_M$ a non-negative cut-off such function such that

$$ \phi=0 \mbox{ in }B_0(M_0), \phi = 1 \mbox{ in } A(0; 2 M_0, M), \ \phi =0 \mbox{ in } B_0(2M)^c,  $$$$ \int_{\R^2} |\nabla \phi_M|^2 \leq C, \quad  \mbox{ independently of  }M.$$

Recall now \eqref{crucial}, and let us estimate:

$$ \int_{\R^2_+} \psi^2 |\nabla \phi|^2 \leq C, $$

However $\psi \geq 1$, and so 

 $$ \oint_{\partial \R^2_+}  e^{v/2} (\psi \phi)^2 \geq  \oint_{\partial \R^2_+ \cap A(0; 2M_0, M)}  e^{v/2} .$$

Since $\oint_{\partial \R^2_+}  e^{v/2} = +\infty$ and $M_0$ is fixed, the above term diverges as $M \to +\infty$. Hence we can choose $M$ so that $Q(\phi \psi)<0$. Since $M_0$ is arbitrary we obtain instability outside any compact set, as claimed.
\end{proof}

\

\section{Blow-up analysis. General properties} \label{s:bu-gen}
\setcounter{equation}{0}

The goal of the rest of the paper is to prove 
Theorem \ref{teo4}. In this section we focus on some general 
properties of blowing-up sequences of solutions to \eqref{ecua-compact}.
In particular we will derive the proof of Theorem \ref{teo4}, {\em (1)}.

\

\begin{proposition}\label{blowanal} The singular set $S$ defined in \eqref{singularset} satisfies  $$S \subset \left\{ p \in \partial \S \,\, : \, \mathfrak{D}(p)\geq 1 \right\}.$$
\end{proposition}

The spirit of the proof is simple: if $p \in S$, then one can rescale around $p$ to obtain a solution on the half-plane, and this is possible only if $\mathfrak{D}(p)\geq 1$, as recalled in the previous section. The non-trivial point here is to be able to rescale and pass to a limit problem even if $p$ is not isolated in $S$. We show that this is possible by choosing carefully a sequence $x_n \in \S$ (not necessarily local maxima) such that  $x_n \to p$ and $u_n(x_n) \to +\infty$. Let us point out that this is a technical novelty even for the classical problem considered in \cite{breme, li-sha}: one can pass to the limit around a singular point \emph{without knowing if the singular set is finite}. 

For this purpose we shall use  Ekeland's variational principle, which we recall below:

\begin{theorem}[see Chapter 5 in \cite{str-book}]\label{Ekel}
Let $(X,d)$ be a complete metric space and consider a function $\varphi:X\to(-\infty,+\infty]$ that is lower semi-continuous, bounded from below and not identical to $+\infty$. Let $\varepsilon>0$ and $\lambda>0$ be given and let $x\in X$ be such that $\varphi(x)\leq\inf_{X}\varphi+\varepsilon$. Then there exists $x_{\varepsilon}\in X$ such that
\begin{enumerate}
\item $\displaystyle{\varphi(x_{\varepsilon})\leq \varphi(x)}$,
\item $\displaystyle{d(x_{\varepsilon},x)\leq \lambda}$,
\item $\displaystyle{\varphi(x_{\varepsilon})<\varphi(z)+\varepsilon \frac{1}{\lambda} d(x_{\varepsilon},z)}$ for every $z\neq x_{\varepsilon}$. 
\end{enumerate}
\end{theorem}

\medskip

\begin{proof}

Let $p$ be a point in $S$. By conformal invariance, we can pass from a neighborhood of $p$ to a domain $B \subset \R^2$, where $B$ denotes $B_0(r) \subset \R^2$, if $p \in int(\S)$, or $B_0^+(r)$, if $p \in \partial \S$. We can choose this conformal map so that $p$ is mapped to $0$.

Let us take a sequence $y_n$  in $B$ such that $y_n\to 0$ and $u_n(y_n)\to+\infty$, and define
$$
\varepsilon_n=e^{-\frac{u_n(y_n)}{2}}.
$$

We apply Theorem~\ref{Ekel} taking $\displaystyle{\varphi=e^{-\frac{u_n}{2}}}$ and $\lambda=\sqrt{\varepsilon_n}$: then there exists a sequence $x_n \in B$ such that

\begin{enumerate}
\item[a)] $u_n(y_n)\leq u_n(x_n)$,
\item[b)] $\displaystyle{d(x_n,y_n)\leq \sqrt{\varepsilon_n}}$,
\item[c)] $\displaystyle{e^{-\frac{u_n(x_n)}{2}}<e^{-\frac{u_n(z)}{2}}+ \sqrt{\varepsilon_n} d(x_n,z)}$ for every $z\neq x_n$. 
\end{enumerate}

As a consequence of a) and b) above, $x_n\to 0$ and $u_n(x_n)\to+\infty$. The idea is that the new sequence $x_n$ is convenient for rescaling and passing to a limit problem.

Now, set:
\beq\label{deltaeps}
\delta_n=e^{\frac{-u_n(x_n)}{2}}\to 0, \  B_n= B_{x_n}( r/2) \cap B,
\eeq
and
%


\beq\label{scalesolution}
v_n(x)=u_n(\delta_n x+x_n)+2\log \delta_n, 
\eeq
which is defined in $\tilde{B}_n=  \frac{1}{\delta_n} B_n$.
Clearly, $v_n(0)=0$. We claim that, given any $R>0$, $\e>0$, 

\beq \label{bound9} v_n(x) \leq \e \ \ \forall \, x  \in \tilde{B}_n,\ |x| < R, \eeq
for sufficiently large $n$. Indeed, we use c) in the choice of the sequence $x_n$ and recall the definition of $\delta_n$; we conclude that if $|z-x_n| < R \delta_n$, then 

$$  (1- \sqrt{\e_n} R )\displaystyle{e^{-\frac{u_n(x_n)}{2}} < e^{-\frac{u_n(z)}{2}}}.$$
	
From this we get:

$$ u_n(z) < u_n(x_n) - 2 \log(1- \sqrt{\e_n} R), \ \mbox{ if } |z-x_n|< R \delta_n.$$

And this implies \eqref{bound9}.

\

In what follows we distinguish two cases:

\

\textit{\underline{Case 1}} $p \in \partial \S$ and, up to a subsequence: 
$$
d(x_n, \Gamma_0^+(r))=O(\delta_n) \quad \mbox{ as } n\to+\infty.
$$


Passing to a subsequence we can assume that $\frac{d(x_n, \Gamma_0^+(r))}{\delta_n } \to t_0 \geq 0$. Then, the function $v_n$ solves 
\begin{equation}\label{ecuablow}
\left\{\begin{array}{ll}
\displaystyle{ -\Delta v_n + 2 \delta_n^2 \tilde{K}_n(\delta_n x+x_n) = 2 K_n(\delta_n x+x_n) e^{v_n}},  \qquad & \text{in } \tilde{B}_n,\\
\displaystyle{\frac{\partial v_n}{\partial n} + 2 \delta_n \tilde{h}_n(\delta_n x + x_n) = 2h_n(\delta_n x+x_n)e^{v_n/2}}, \qquad  &\text{on } \tilde{\Gamma}_n,
\end{array}\right.
\end{equation}
where $\tilde{\Gamma}_n$ is the straight portion of $\partial \tilde{B}_n$.

Taking \eqref{bound9} into account, by Harnack type inequalities (see Lemma A.2 in \cite{JostWZZ}), $v_n$ is uniformly bounded in $L^{\infty}_{loc}(\R \times (-t_0,+\infty))$. Therefore, up to subsequence,
\beq\label{C2conver}
v_n\to v \quad \mbox{ in } C^2_{loc}(\R \times (-t_0,+\infty)), 
\eeq
which is a solution of the equation

\begin{equation}\label{ecualimit1}
\left\{\begin{array}{ll}
\displaystyle{-\Delta v = 2 K(0) e^v},  \qquad & \text{in } \R \times (-t_0,+\infty),\\
\displaystyle{\frac{\partial v}{\partial n} = 2h(0)e^{v/2}}, \qquad  &\text{if }t=-t_0.
\end{array}\right.
\end{equation}

If $K(0)<0$ and $h(0)>0$, the latter problem  admits solutions only if $\mathfrak{D}(0) \geq 1$, see Section 4. 

\
\

\textit{\underline{Case 2}}
$$
\frac{d(x_n,\Gamma_0^+(r))}{\delta_n} \to+\infty, \quad \mbox{ as } n\to+\infty.
$$

In this situation the rescaled domains $\tilde{B}_n$ invade all of $\R^2$. Hence, 
reasoning as before, up to subsequence we have 
$$
v_n\to v \quad \mbox{ uniformly in } C^2_{loc}(\mathbb{R}^2), 
$$
which is a solution of the equation
\begin{equation}\label{ecualimit2}
\displaystyle{-\Delta v = 2 K(0) e^{v}},  \qquad \text{in $\mathbb{R}^2$.}
\end{equation}

If $K(0)<0$, is it well known, via Liouville's formula, that problem \eqref{ecualimit2} does not admit any solution.

\end{proof}

\begin{remark} \label{bordo}

In the above proof, in Case 1, we arrive to a solution of \eqref{ecualimit1}. Moreover $v(0)=0$ and, by \eqref{bound9}, $v(x) \leq 0 $ for all $x \in \R \times [-t_0,+\infty)]$. By the maximum principle this implies that actually $t_0=0$. That is, one obtains the more 
precise conclusion $d(x_n, \Gamma_0^+(r)) = o (\delta_n)$.

\end{remark}

We now state and prove a couple of lemmas for later use. 

\begin{lemma} \label{u^-}There holds:

\begin{enumerate}

\item[a)] For any $q \in (1,2)$, $\int_\S |\nabla u_n^-|^q < O(1)$. 

\item[b)] If, moreover, $\chi_n \leq 0$, then $\int_\S |\nabla u_n^-|^2 < O(1)$.

\item[c)] For any $K \subset \S$ compact set with $K \cap S = \emptyset$, we have that:

$$ \sup_{K} u_n - \inf_K u_n = O(1) , \ \ \int_K |\nabla u_n|^2 = O(1).$$

\end{enumerate}

\end{lemma}

Estimates like \emph{c)} are commonly used in Liouville type problems, and  are usually proved by using a Green's representation argument. However, in our case this is not possible unless the total mass is finite. Here we give a different proof, which is based on estimates for $u_n^-$ and on local regularity arguments. 

The estimates of $u_n^-$ follow the idea of \cite{dmlsr}, where the Kato inequality is used. The problem is that here it is not so clear which boundary condition is satisfied by $u_n^-$, hence we use a smooth approximation of the function $u \mapsto u^-$.

\begin{proof} Define $w: \R \to \R$ a $C^2$ function such that:
$$ \left \{ \begin{array}{ll} w(u)=u & u \leq 0, \\ w(u)=1 & u \geq 1, \\ w'(u) \geq 0 & u \in \R, \\ w''(u) \leq 0 & u \in \R. \end{array} \right.$$

Then the function $w_n=w(u_n)$ satisfies

\begin{equation} \label{w(u)} \left \{ \begin{array}{ll}\dis  -\Delta w_n = -w''(u_n)|\nabla u_n|^2 - 2 w'(u_n) (\tilde{K}_n + |K_n| e^{u_n}), & x \in \Sigma, \\ \dis \  \frac{\partial w_n}{\partial \nu} + 2 w'(u_n) \tilde{h}_n = 2 w'(u_n) h_n e^{u_n/2}, & x \in \partial \S. \end{array} \right. \end{equation}

Just by integrating we find that
$$
0 \leq  \int_\S -w''(u_n)|\nabla u_n|^2 \leq O(1).
$$

As a consequence, the right-hand side of the first equation in \eqref{w(u)} is bounded in $L^1$, whereas the boundary data are bounded in $L^\infty$. By elliptic regularity estimates, $\int_\S |\nabla w_n|^q < O(1)$ for any $q \in (1,2)$. This implies \emph{a)}.

\medskip We now prove \emph{b)}. We multiply \eqref{ecua-compact} by $u_n^-$ and integrate, to get 
\begin{align*} \int_\S  |\nabla u_n^-|^2 + 2 u_n^- (\tilde{K}_n + |K_n| e^{u_n})]   =  2 \oint_{\partial \S}  u_n^- (h_n e^{u_n/2} - \tilde{h}_n). \end{align*}

Clearly the functions $u_n^- e^{u_n}$, $u_n^- e^{u_n/2}$ are bounded in $L^\infty$. Hence, 
$$  \int_\S  |\nabla u_n^-|^2  = -2 \left [ \int_{\Sigma}   \tilde{K}_n u_n^- + \oint_{\partial \S}  \tilde{h}_n u_n^- \right ] + O(1).$$  We are able to estimate the right hand side by the assumption on $\chi_n$. We split $ u_n^- = (u_n^- - \fint_\S u_n^-) +  \fint_\S u_n^-$, and, making use of \emph{a)}, we obtain:
$$ \left | \int_\S    \tilde{K}_n(u_n^- - \fint_\S u_n^-) \right | \leq O(1), \ \ \left | \oint_{\partial \S}   \tilde{h}_n(u_n^- - \fint_\S u_n^-) \right | \leq O(1).$$
Moreover,
$$
-2  \left [ \int_\S \tilde{K}_n \fint_\S u_n^- + \oint_{\partial \S} \tilde{h}_n \fint_\S u_n^-\right  ] = -2 \chi_n \fint_\S u_n^- \leq 0, 
$$
yielding the assertion. 

Finally, let us show \emph{c)}. We can assume that $K = \overline{B}_p(r)$ where $B_p(4 r) \cap S = \emptyset$ with $p\in\Sigma$ and:

\begin{enumerate}
	\item either $B_p(4 r) \subset int(\S)$,
	\item or $p \in \partial \S$.
	
\end{enumerate}

Via a conformal map we can pass to a problem in $B_0(4r) \subset \R^2$ or $B_0^+(4r) \subset \R^2_+$.

Observe that in any case $u_n \leq C$ in $B_p(2r)$ for some $C>0$. Recall that as in a), we have that the function $v_n = \min \{u_n , C \}$  satisfies  $\int_{\S} |\nabla v_n|^q = O(1)$ for all $q \in [1,2)$. As a consequence, the function $\tilde{u}_n= u_n - \fint_{B_p(2r)} u_n$ is bounded in $L^{q'}(B_p(2r))$ for any $q'>1$.

In case (1) we have that $\tilde{u}_n$ solves:
	
$$
\begin{array}{ll}
-\Delta \tilde{u}_n = f_n,  \qquad & \text{in }B_0(2r),
	\end{array}
$$
with $f_n$ bounded in $L^{\infty}$. We conclude then by local regularity estimates (see for instance \cite[Theorem 8.17 and Theorem 8.32]{gilbarg}).

In case 2 we are led to the problem:

$$
\left\{\begin{array}{ll}
\displaystyle{-\Delta \tilde{u}_n = f_n,}  \qquad & \text{in $B_0^+(2r)$,}\\
\displaystyle{\frac{\partial \tilde{u}_n}{\partial n} =g_n}, \qquad  &\text{on $\Gamma_0^+(2r)$,}
\end{array}\right.
$$
with $f_n$, $g_n$ bounded in $L^{\infty}$. We conclude by local regularity estimates for the Neumann problem (see for instance \cite[Theorem 5.36 and Lemma 5.51]{lieberman}).

\end{proof}

\

%
%
%
%
%

The following lemma gives a Pohozaev type identity, depending on an arbitrary field $F$, which we state and prove for the sake of completeness.

\begin{lemma} \label{lema4} Let $u$ be a solution of 
$$ -\Delta u +2 \tilde{K}= 2 K(x) e^u,  \qquad  \text{in }\Sigma.$$
Then, given any vector field $F: \Sigma \to T \Sigma$, there holds:
\begin{align*}
 \oint_{\partial \Sigma} [4 K(x) e^u ( F\cdot \nu) + 2 (\nabla u \cdot \nu) (\nabla u \cdot F) - |\nabla u|^2 F \cdot \nu]  \\ = \int_{\Sigma} [4 \tilde{K} \nabla u \cdot F +4 e^u ( \nabla K \cdot F + K\  \nabla \cdot F ) + 2 DF(\nabla u, \nabla u)  - \nabla \cdot F |\nabla u|^2]. \end{align*}

\end{lemma}

\begin{proof}
We will make use of the following basic identity:

$$ 2 \Delta u ( \nabla u \cdot F) = \nabla  \cdot ( 2 (\nabla u \cdot F) \nabla u - |\nabla u|^2 F ) - 2  DF(\nabla u, \nabla u) + |\nabla u|^2 \nabla \cdot F.$$
With this identity at hand, we just multiply the equation by $2 \nabla u \cdot F$ and integrate, using the divergence theorem. Take into account also that, again by the divergence theorem,
$$ 4 \int_\Sigma K e^u \nabla u \cdot F = 4 \left[ \oint_{\partial \Sigma} K e^u F\cdot \nu  - \int_{\Sigma} e^u ( \nabla K \cdot F + K\  \nabla \cdot F )   \right].$$

\end{proof}

\

\section{Blow-up with unbounded mass}\setcounter{equation}{0}

In this section we  address the question of infinite-mass blow-up. As we shall see later, this is the only possible blow-up scenario if $\chi_n \leq 0$, and it is a completely new phenomenon in this type of problems. 
Let us set:
$$ \rho_n = \oint_{\partial \Sigma} h_n e^{u_n/2} = \int_{\Sigma} |K_n| e^{u_n} + O(1) \to +\infty.$$

Hence, up to a subsequence, we obtain that

\begin{enumerate}

\item $ \rho_n^{-1} h e^{u_n/2} \rightharpoonup \sigma$,

\item $\rho_n^{-1} |K| e^{u_n} \rightharpoonup \xi,$ 

\end{enumerate}
where $\sigma$, $\xi$ are unit positive measures defined on $\partial \S$ and $\S$, respectively, and the above is weak convergence of measures. 

The next proposition implies {\em (3) a)} of Theorem \ref{teo4}.

\begin{proposition} \label{lema3} $\xi|_{int(\Sigma)}=0$ and  $ \xi|_{\partial \Sigma}= \sigma $.
\end{proposition}

\begin{proof}

Fix $\phi \in C^2(\Sigma)$, multiply equation \eqref{ecua} by $\phi$ and use Green's formula, to obtain:
$$ 2 \oint_{\partial \Sigma} [h_n e^{u_n/2} - \tilde{h}_n]  \phi -2 \int_{\Sigma} [\hat{K}_n \phi + |K_n|e^{u_n} \phi] =  \int_{\Sigma} u_n \Delta \phi  + \oint_{\partial \S} \frac{\partial \phi}{\partial \nu} u_n.$$

We now estimate the right hand side taking into account the positive and negative parts of $u_n$, 
where we recall that in our notation $u^+ = \max \{u, 0\}$ and $u^- = \min \{u, 0\}$: 
$$\left | \int_{\Sigma} u_n^+ \Delta \phi   + \oint_{\partial \S} \frac{\partial \phi}{\partial \nu} u_n^+ \right | \leq C \int_\S u_n^+ + \oint_{\partial \S} u_n^+ = o (\rho_n).$$

Moreover,
$$\int_{\Sigma} u_n^- \Delta \phi   + \oint_{\partial \S} \frac{\partial \phi}{\partial \nu} u_n^- = \int_{\Sigma} (u_n^- - \fint_\S u_n^-) \Delta \phi   + \oint_{\partial \S} \frac{\partial \phi}{\partial \nu} (u_n^- - \fint_\S u_n^-),$$

and
$$ \int_\S |u_n^- - \fint_\S u_n^-| + \oint_{\partial \S} |u_n^- - \fint_\S u_n^-| \leq C \int_\S | \nabla u_n^-|^q,$$
for any $q\in (1,2)$, and this last quantity is bounded by Lemma \ref{u^-}, \emph{a)}.

Since $\phi$ is arbitrary, we can conclude.
\end{proof}

Observe that $supp \, \sigma \subset S \subset \{p \in \partial \S: \ \mathfrak{D}(p) \geq 1\}$ by Proposition \ref{blowanal}. In what follows we will show that $\supp \sigma \subset \{ p \in \partial \S: \ \mathfrak{D}_{\tau}(p)=0\}$. This will be accomplished by making use of the Pohozaev-type identity given in Lemma \ref{lema4} on fields that are tangential to $\partial \S$. However, here the question is delicate since the support of $\sigma$ need not be finite. Moreover, and more importantly, we do not have any control of the asymptotic behavior of the Dirichlet energy of the solutions. 

The idea is to apply Lemma \ref{lema4} to holomorphic fields $F$. In this way, the Cauchy-Riemann equations imply that the terms involving the Dirichlet energy vanish. For this, we will need to first pass to an analytic setting via a conformal map.

\begin{proposition}\label{p:int-parts}  $\supp \sigma \subset \{p \in \partial \Sigma:\ \mathfrak{D}_\tau(p)=0 \}$. \end{proposition} 

\begin{proof}
	
Let $\Lambda_0$ be a connected component of $\partial \S$ such that $\sigma|_{\Lambda_0} \neq 0$, and 
consider a smooth neighborhood $U$. By the Uniformization Theorem for annuli (see for instance \cite[IV.7]{farkas}) we can pass via a conformal map to a problem in the annulus $A(0; r, 1)$ for some $r>0$. That is, we need to consider:

$$
\left\{\begin{array}{ll}
\displaystyle{-\Delta_g u_n +2 \tilde{K}_n(x)= 2 K_n(x) e^{u_n}},  \qquad & \text{in } A(0; r, 1), \\
\displaystyle{\frac{\partial u_n}{\partial n} +2 \tilde{h}_n(x)  = 2h_n(x)e^{u_n/2}}, \qquad  &\text{on
	$\Lambda_0 $,}
\end{array}\right.
$$
where $g$ is a metric conformal to the standard one $g_0$, that is, $g = e^v g_0$. Here we identify $\Lambda_0 = \{|x|=1\}$. 

As a consequence, the function $\hat{u}_n = u_n + v$ satisfies the following equation with respect to $g_0$:

\begin{equation}\label{anillo}
\left\{\begin{array}{ll}
\displaystyle{-\Delta \hat{u}_n +2 \hat{K}_n(x)= 2 K_n(x) e^{\hat{u}_n}},  \qquad & \text{in } A(0; r, 1),\\
\displaystyle{\frac{\partial u_n}{\partial n} +2 \hat{h}_n(x)  = 2h_n(x)e^{\hat{u}_n/2}}, \qquad  &\text{on
	$\Lambda_0 $,}
\end{array}\right.
\end{equation}
for some smooth functions $\hat{K}_n$, $\hat{h}_n$. In what follows we can consider problem \eqref{anillo} in a flat annulus and prove the statement of Proposition \ref{p:int-parts}.

\

Take any analytic real function $f$ defined in $\Lambda_0$, and consider $F: \Lambda_0 \to \mathbb{C}$ defined by $F(p)= f(p) \tau(p)$, where $\tau(p)$ is the tangent unit vector. By analytic continuation we can extend $F$ to a holomorphic function $F: A(0;r,1) \to \mathbb{C}$, by taking $r$ closer to $1$ if necessary. We define $\tilde{F}(p)= F(p) \phi(p)$, where $\phi$ is a cut-off such that $\phi = 1$ in $A(0; r_0, 1)$ and $\phi=0 $ in $A(0; r, r_1)$ with $r<r_1 <r_0$. We apply Lemma \ref{lema4} to the field $F$ to obtain:
\begin{align*} &  \oint_{\Lambda_0}  4 f (h_n e^{\hat{u}_n/2} - \hat{h}_n)  (\hat{u}_n)_\tau  \\ & = \int_{A(0;r_0,1)}  [4 \hat{K}_n \nabla \hat{u}_n \cdot F + 4 e^{\hat{u}_n} ( \nabla K_n \cdot F + K_n\  \nabla \cdot F)] + O(1). \nonumber  \end{align*}

Observe that here we are using Lemma \ref{u^-}, c). Let us also point out here that $ 2 DF(\nabla \hat{u}_n, \nabla \hat{u}_n) - \nabla \cdot F |\nabla \hat{u}_n|^2=0$ by the Cauchy-Riemann equations.

We now get rid of 
$$ \int_{A(0;r_0,1)}  4 \hat{K}_n \nabla \hat{u}_n \cdot F= \int_{A(0;r_0,1)}  4 \hat{K}_n (\nabla \hat{u}_n^+ + \nabla \hat{u}_n^-) \cdot F.$$

The term $ \int_{A(0;r_0,1)}  4 \hat{K}_n \nabla \hat{u}_n^- \cdot F$ is bounded by Lemma \ref{u^-}, whereas the 
one with the positive part of $\hat{u}_n$ can be estimated via an integration by parts:

\begin{align*} \left | \int_{A(0;r_0,1)}  4 \hat{K}_n \nabla \hat{u}_n^+ \cdot F \right | \leq  \left | \oint_{\Lambda_0} 4 \hat{K}_n \hat{u}_n^+ F \cdot \nu - \int_{A(0;r_0,1)} \hat{u}_n^+ \nabla \cdot (\hat{K}_n F) \right | + O(1) \\ \leq C \left ( \oint_{\Lambda_0} \hat{u}_n^+ + \int_{A(0;r_0,1)} \hat{u}_n^+ \right ) + O(1)=o( \rho_n). \end{align*}

In the same way we can estimate the term:

$$\oint_{\Lambda_0}  f  \hat{h}_n  (\hat{u}_n)_\tau = \oint_{\Lambda_0}  f  \hat{h}_n  (\hat{u}_n^+ + \hat{u}_n^-)_\tau. $$

Indeed,

\begin{align*}
\oint_{\Lambda_0}  f  \hat{h}_n  (\hat{u}_n^-)_\tau = \oint_{\Lambda_0}  f  \hat{h}_n  \left (\hat{u}_n^- - \fint_{\Lambda_0} u_n^- \right )_\tau = \\ - \oint_{\Lambda_0} ( f  \hat{h}_n)_{\tau}  \left (\hat{u}_n^- - \fint_{\Lambda_0} u_n^- \right ) = O(1), 
\end{align*}
by Lemma \ref{u^-}. Moreover,

$$ \oint_{\Lambda_0}  f  \hat{h}_n  (\hat{u}_n^+)_\tau = - \oint_{\Lambda_0}  (f  \hat{h}_n)_{\tau} \hat{u}_n^+ = o(\rho_n).$$

Integrating by parts we find
$$ 4  \oint_{\Lambda_0}  h_n f e^{\hat{u}_n/2} (\nabla \hat{u}_n \cdot F)= -8  \oint_{\Lambda_0}  (h_n f)_{\tau} e^{\hat{u}_n/2} .$$

Then,

$$ -8 \oint_{\Lambda_0}   ((h_n)_\tau f + h_ n f_\tau) e^{\hat{u}_n/2} = \int_{A(0;r_0,1)} (4 \nabla K_n \cdot F + 4 K_n \nabla \cdot F )  e^{\hat{u}_n}  + o(\rho_n).$$

Observe now that on $\Lambda_0$, $\nabla \cdot F= 2 f_{\tau}$. By Proposition \ref{lema3}, we can divide by $\rho_n$ and pass to the limit to obtain:

$$
 -8 \oint_{\Lambda_0}  \left( \frac{h_\tau}{h} f + f_\tau \right) \, d \sigma = - \oint_{\Lambda_0} 
 \left( 4 f \frac{K_\tau}{K} + 8 f_\tau \right)  \, d \sigma.
$$

Recall that in the support of $\sigma$ we have the inequality $\mathfrak{D}(p) \geq 1$, which implies that $h(p)$ is positive: this allows us to write $h$ 
in the denominator. 

The terms in $f_\tau$ cancel and we can rewrite this expression as:

$$
 \oint_{\Lambda_0} \left ( 2  \frac{h_\tau}{h} -   \frac{K_{\tau}}{K}   \right ) f \, d \sigma= 0.
$$

Let us define the measure $\mu =  \left ( 2  \frac{h_\tau}{h} -   \frac{K_{\tau}}{K}   \right ) \sigma$. Then, we have obtained that $ \oint_{\Lambda_0} f \, d\mu=0$ for any analytic real function $f$. Since the analytic functions form a dense subset of the space of continuous functions, we conclude that $\mu=0$. 

Finally, notice that
 
$$ 
 2  \frac{h_\tau}{h} -   \frac{K_{\tau}}{K} = 2 \frac{\mathfrak{D}_\tau}{\mathfrak{D}}.
$$
Hence $\mathfrak{D}_\tau(p)=0$ for any $p \in \supp \sigma$.
\end{proof}

\

\section{Blow-up with bounded mass or  Morse index} \label{s:bu-bdd}
\setcounter{equation}{0}

In this section we consider the case in which the sequence $u_n$ has either bounded mass or bounded Morse index, and we will exploit this information to give a more complete description of the blow-up phenomena, proving {\em (2)} and {\em (4)} in 
Theorem \ref{teo4}. The key ingredient here is that, if $\mathfrak{D}(p)>1$, the only limit profiles with bounded mass or bounded Morse index are the bubbles in the form \eqref{bubble}.

Let us start with the following fact:

\begin{lemma}\label{minimalmass} Assume that $\int_{\S} e^{u_n}$ is bounded in $n$. Then, 
	$$
	S \subset \left\{p\in\partial\Sigma: \mathfrak{D}(p)> 1 \right\}.
	$$
\end{lemma}

\begin{proof}
	Let $p$ be a point in $S$, and consider the function $v_n$ defined in \eqref{scalesolution}. By Fatou's lemma, the classification of the entire solutions on the upper half-plane and Remark~\ref{bordo}, we find that for some $r>0$
	
	$$
	C \geq \lim_{n\to+\infty} \int_{B^+_{p}(r)} |K_n|e^{u_n} = \lim_{n\to+\infty} \int_{\tilde{B}_n} |K_n|e^{v_n} \geq \int_{\mathbb{R}^2_+} |K(p_0)|e^{v}, 
	$$
	where $\tilde{B}_n=B_{\frac{x_n-p}{\delta_n}}^+(\frac{r}{\delta_n})$ and $v$ is an entire solution of the limit problem \eqref{ecualimit1}. Observe that if $\mathfrak{D}(p)=1$, 
	$v$ is of the form \eqref{profile2} (see Section 4), so the upper bound on the volume 
	is violated. Therefore $\mathfrak{D}(p)>1$. 
	
\end{proof}

The main result of the section is the following proposition, which completes the proof of Theorem \ref{teo4} and includes the a\-sym\-ptotic behavior of the blowing-up sequences of solutions near the singular points.

\begin{proposition}\label{p:6.3}
	Let $u_n$ be a blowing-up sequence  (namely, $ \sup_\Sigma  u_n  \to +\infty$) of solutions to \eqref{ecua-compact}, under the conditions of Theorem \ref{teo4}. Assume also that:
	
\begin{equation} \label{cond} \mbox{ either} \int_{\S} e^{u_n}  \ \mbox{ or } \ ind(u_n) \mbox{ is bounded.} \end{equation}
	
	Then $S= S_0 \cup S_1$, with:
	$$S_0 \subset  \{p \in \partial \Sigma:\ \mathfrak{D}(p) = 1,\ \mathfrak{D}_{\tau}(p) =0\},$$
	$$S_1 =\{p_1,\dots p_m\}  \subset  \{p\in\partial\S : \mathfrak{D}(p) > 1 \}.$$
	
	Moreover, the following asymptotics hold:
	
	\
	
\begin{enumerate}
	\item[a)] 	If $p \in S_0$, there exists $x_n \to p$, $\lambda_n \to 0$, $R_n \to +\infty$ such that $d(x_n, \partial \S) = o (\lambda_n)$ and
	\begin{equation} \label{eq:profile-6}
	u_n(y) = v_{\lambda_n}(0, t) + o(1), \ \ y \in B_{x_n}(R_n \delta_n),
	\end{equation}
	where $v_\lambda$ is the $1-D$ solution given in \eqref{profile2}, and $t= d(y, \partial \S)$.

	\item[b)] If $p \in S_1$, there exists $(s_n,t_n)\to p$ and $\lambda_n\to 0$
	\beq\label{profile}
	u_n(s,t) =  2 \log\left( \frac{2\lambda_n}{(s-s_n)^2+(t+\mathfrak{D}(p)\lambda_n)^2-\lambda^2_n} \right)+O(1),
\eeq
in $B^+_{p}(r)\subset\mathbb{R}^2$ with $d(p,p_j)>2r$ for any $j=1,\ldots,m$, where $(s,t)$ is an isothermal coordinate centered at $p$. Moreover, 
	\beq\label{est1}
	\nabla u_n(x) = -4 \frac{x-p}{|x-p|^2} + o\left(\frac{1}{|x-p|^2}\right), \ \ \mbox{ in $B_{p}^+(r)\setminus B_{p}^+(\lambda_n\log\frac{1}{\lambda_n}).$}
	\eeq

	\item[c)] Finally, if $\chi_n \leq 0$,  then $S_1$ is empty.
\end{enumerate}
	
\end{proposition}

This result follows from the lemmas below.

\begin{lemma}\label{l:vol-fin-S1}
	Under assumption \eqref{cond}, the set $S_1$ is finite. Moreover, let $p \in S_1$ and consider a blow-up profile $v$ at $p$ constructed as in the 
	proof of Proposition \ref{blowanal}. Then $v$ has finite volume.  
\end{lemma}

\begin{proof}
	Recalling \eqref{ecua-compact}, the Morse index bound is related to the 
	quadratic form \eqref{quadratic}. Recall also that the limit profile $v$ is defined as the limit of the sequence 
	$$
	v_n(x) = u_n(\delta_n x + x_n) + 2 \log \delta_n; \qquad \delta_n = e^{-u_n(x_n)/2}. 
	$$
	We claim that \eqref{finitemass} is satisfied for $v$. Otherwise, by Fatou's lemma, $\int_{\S} e^{u_n}$ is unbounded. Moreover, by Theorem  \ref{t:app-index}  {\em b)}   we can find compactly-supported functions 
	(with disjoint supports) $\psi_1, \dots, \psi_{m+1}$  such that 
	$$ a_i := \int_{\R^2_+}  [|\nabla (\psi_i)|^2 + 2 e^v (\psi_i)^2]  - \int_{\partial \R^2_+} \frac{1}{\mathfrak{D}(p)}  e^{v/2} (\psi_i)^2 < 0$$
	for all $i = 1, \dots, m+1$.  
	
	Define then 
	$$
	\psi_{i,n} = \psi_i \left( \frac{x-x_n}{\delta_n} \right). 
	$$
	By scaling variables, recalling the expression in \eqref{quadratic}, it is then easy to see that 
	$$
	Q_n(\psi_{i,n}) = a_i + o_n(1) < 0; \qquad \quad  i = 1, \dots, m+1. 
	$$
	Since also the $\psi_{i,n}$'s have disjoint supports, we deduce that $ind(u_n) \geq m+1 $ for $n$ 
	large. In this way we arrive to a contradiction with \eqref{cond}. 
	
	\
	
We know now that for any point in $S_1$ we can construct a blow-up profile $v$ 
	that satisfies \eqref{finitemass}, and in particular its Morse index is equal to $1$. Moreover, by \eqref{eq:beta},

	$$
	\liminf_{n \to +\infty} \int_{\S} |K_n| e^{u_n} \geq \int_{\mathbb{R}^2_+} |K(p_0)| e^v= 2\pi \left(\frac{ h(p_0)}{\sqrt{h^2(p_0)+K(p_0)}}-1 \right)>\delta,
	$$
	for a suitable $\delta>0$. Hence condition \eqref{cond} implies that $S$ is finite.
	
	\end{proof}

\

\begin{lemma}\label{key}
	Let $p \in S_1$. Then there exists  fixed constants $r, C > 0$ such that 
	$$
	\int_{B_{p}(r)\cap\Sigma} e^{u_n} \leq C; \qquad \oint_{B_{p}(r)\cap\partial\Sigma} e^{u_n/2} \leq C, \quad u_n \to -\infty \mbox{ on } \partial B_p(r) \cap \Sigma.
	$$
	
\end{lemma}

\begin{proof} 
	We will follow a modification of  the strategy in \cite{li-sha}, and hence we will be  sketchy in some parts. 
	First, we can  conformally deform a geodesic ball centred at $p$ into a planar half-ball
	$B_0^+(r_0)$, reducing ourselves to the same equation as in \eqref{anillo}. 
	It is then convenient to divide the proof into several steps. 
	
	%
	%
	%
	%
	%
	%
	
	\
	
	\noindent {\bf Step 1: blowing a first bubble.}  We already know from Lemma \ref{l:vol-fin-S1} that 
	the subset $S_1$ in the blow-up set $S$  is finite and therefore, by sub-harmonicity, for $n$ large (and up to a subsequence) $u_n$ has 
	local maxima  $x_n$ on $\Gamma^+_0(r)$, converging to $p$.

	We rescale the solutions $u_n$ as follows 
	$$
	v_n(x) = u_n(x_n + \d_n x) + 2 \log \d_n; \qquad \qquad e^{-2 \d_n} = u_n(x_n). 
	$$
	Reasoning as in the proof of Proposition \ref{blowanal}, 
	we can find $R_n \to + \infty$ (slowly) 
	such that $\|v_n - v\|_{C^2(B_0^+(2 R_n))} \to 0$, with  
	$v$ a solution of \eqref{ecualimit1}. By the classification results given in Section 4, $v$ must be of the form 
	$$
	v(s,t) = 2 \log \left(  \frac{2 \l}{- \l ^2 + (s-s_0)^2 + (t+t_0)^2} \right)
	$$
	for some $\l > 0$, $s_0 \in \R$ and $t_0$ determined by $t_0 = \mathfrak{D}(p) \, \l $. Notice that, at infinity 
	\begin{equation}\label{eq:asy-ovv}
	v(s,t) \simeq - 4 \log \left| (s,t) - (s_0, -t_0) \right|. 
	\end{equation}
	We also recall that this solution has Morse index one, by Theorem \ref{t:app-index}.

	\
	
	\noindent {\bf Step 2: blowing (possibly) other bubbles.} Letting $r$ be as in the previous step and 
	$R_n \to + \infty$ as in Step 1, we consider the  maximization problem 
	$$
	\sup_{R_n \d_n \leq |x-x_n| \leq r/2} \left( u_n + 2 \log |x-x_n| \right). 
	$$
	If this supremum tends to infinity, reasoning as in the proof of Lemma 4 in \cite{li-sha}, 
	one could again rescale $u_n$  
	near a maximum point of the function $u_n + 2 \log |x-x_n|$ to obtain 
	another limiting profile. Since all limiting profiles have the expression 
	of the above function $v$ and have Morse index equal to one and fixed mass, \eqref{cond} implies that continuing the procedure 
	there is a finite number of bubbles. Therefore, we can find an integer $k$ for $1 \leq k \leq m$, 
	sequences of boundary points $x_n^i \in B_0^+(r)$, $i = 1, \dots, k$, sequences $(\d_n^i)_n$, $\d_n^i \to 0$ for $i = 1, \dots, k$
	sequences $(R_n^i)_n$, $R_n^i \to \infty$, $i = 1, \dots, k$ and a fixed constant $C$ with the following properties

	\begin{description}
		\item[(i)] $v_n^i(x) := u_n(x^i_n + \d^i_n x)+ 2 \log \d_n^i \to v$ \quad in $C^{2}(B_0^+(R^i_n))$; 
		
		\medskip
		
		\item[(ii)] $B^+_{x^i_n}(4 R^i_n \d^i_n) \cap B^+_{x^j_n}(4 R^j_n \d^j_n) = \emptyset$ \quad for $i \neq j$; 
		
		\medskip
		
		\item[(iii)] $
		\sup_{B^+_{0}(r/2) \setminus \cup_{i=1}^k B^+_{x^i_n}(R^i_n \d^i_n)} \left( u_n + 2 \log \min_i |x-x^i_n| \right) \leq C$.
	\end{description}

	\
	
	\noindent {\bf Step 3: Harnack inequality.} For $s^i_n \in [2 R^i_n \d ^i_n, 1/8 \min_{j \neq i} 
	|x^i_n - x^j_n|]$, $s^i_n \leq \frac{r}{8}$, one can consider the rescaled function 
	$w^i_n$ given by 
	$$
	w^i_n(x) = u_n(x^i_n+s^i_n x) + 2 \log s^i_n; \qquad \qquad x \in B^+_0(4) \setminus B^+_0(1/2). 
	$$
	By the above bound {\bf (iii)} on $u^i_n$, $w^i_n$ is uniformly controlled from above, and 
	satisfies an equation with uniformly  bounded data in $B^+_0(4) \setminus B^+_0(1/2)$. 
	Reasoning as for \eqref{C2conver}, we can deduce 
    a Harnack inequality for $e^{w^i_n}$ in 
	$B^+_0(2) \setminus B^+_0(1)$, i.e., $w_n$ has uniformly  bounded oscillation there. 
	Define  the (semi-circular) average 
	$$
	\overline{u}^i_n(s) := \frac{1}{|\partial^+ B_{{x}^i_n}(s)|} \oint_{\partial^+ B_{{x}^i_n}(s)} u_n. 
	$$
	The Harnack inequality implies that there exists a fixed $C > 0$ such that 
	$$
	u_n(x) \leq \overline{u}^i_n(s)  + C, \qquad x \in B^+_{{x}^i_n}(2 s) \setminus B^+_{{x}^i_n}(s/2), 
	$$  
	provided that $s \in [2 R^i_n \d ^i_n, 1/8 \min_{j \neq i}   |x^i_n - x^j_n|]$.

	\
	
	\noindent {\bf Step 4: local radial decay.}
	Let us assume first that there exists $x^j_n \neq x^i_n$ as in Step 2 such that $|x^i_n - x^j_n| \to 0$. 
	We will show that the functions $u_n$ (properly rescaled) keep the profile as in 
	\eqref{eq:asy-ovv} for $|x-x^i_n| \leq O (\min_{j \neq i}   |x^i_n - x^j_n|)$. 
	This will imply in particular the finiteness of accumulation of  volume up 
	to that scale.

	If  $\nu_s$ stands for the  outer unit normal to $\partial^+ B_{{x}^i_n}(s)$, one has the formula 

	\begin{equation}\label{eq:der-average-2}
	|\partial^+ B_{{x}^i_n}(s)| \frac{d}{ds} \overline{u}^i_n(s) =  \oint_{\partial^+ B_{{x}^i_n}(s)} 
	\frac{\partial u_n}{\partial \nu_s}. 
	\end{equation}

	Fix now a number $\d$ small and positive: we claim that 
	\begin{equation}\label{eq:claim-blow-up}
	\overline{u}^i_n(s)  \leq  \overline{u}^i_n(R^i_n \d ^i_n)  - (4-\d) \left( \log s - \log (R^i_n \d ^i_n) \right) 
	\end{equation}
	for all $s \in [2 R^i_n \d ^i_n, 1/8 \min_{j \neq i} 
	|x^i_n- x^j_n|]$. Assuming by contradiction that this is false, 
	define  $t_i^n$ to be the infimum of the radii $s$ such that 
	$$
	\overline{u}^i_n(s)  >  \overline{u}^i_n(R^i_n \d ^i_n)  - (4-\d) \left( \log s - \log (R^i_n \d ^i_n) \right). 
	$$
	Notice that at this first value $t^i_n$ we must also have
	$$
	t_i^n \frac{d}{ds}|_{s=t^i_n} \overline{u}^i_n(s) \geq \d - 4. 
	$$
	\eqref{eq:der-average-2},  the fact that $ |\partial^+ B_{{x}^i_n}(s)| = 
	\pi \, s$ and the latter formula imply 
	$$
	\oint_{\partial^+ B_{{x}^i_n}( t_i^n)} 
	\frac{\partial u_n}{\partial \nu_{t_i^n}} \geq (\d - 4) \pi (1 + O((t_i^n)^2)).  
	$$
	On the other hand, from the convergence in Step 1 and the limiting 
	behavior in \eqref{eq:asy-ovv} we find that 
	$$
	\oint_{\partial^+ B_{{x}^i_n }( R^i_n \d ^i_n)} 
	\frac{\partial u_n}{\partial \nu_{R^i_n \d ^i_n}} \to - 4 \pi 
	\quad \hbox{ as } n \to + \infty. 
	$$
	Integrating \eqref{ecua-compact} and using the last two formulas, together with the 
	fact that $t_i \to 0$, we obtain 
	\begin{eqnarray}\label{eq:int-parts-bdry} \nonumber 
	2 \int_{B^+_{{x}^i_n}( t_i^n) \setminus B^+_{{x}^i_n }( R^i_n \d ^i_n)} (K_n e^{u_n} - \tilde{K}_n u_n) 
	+ 2 \oint_{\Gamma^+_{{x}^i_n}( t_i^n) \setminus \Gamma^+_{{x}^i_n }( R^i_n \d ^i_n)} (h_n e^{u_n/2} - \tilde{h}_n u_n) \\ 
	\leq - \d \pi 
	+ o_n(1). 
	\end{eqnarray}
	
	However, by the definition of $t_i$ and by the Harnack inequality,
	in the above regions we have  that 
	\begin{equation}\label{eq:ineq-exp}
	e^{u_n}(x) \leq C (\d ^i_n)^{-2} (R^i_n)^{-4} \left( \frac{R^i_n \d ^i_n}{|x|} \right)^{4-\d} \qquad 
	|x| \in  [ R^i_n \d ^i_n, t_i^n ]
	\end{equation}
	and moreover, by the asymptotics of $u_n$ 
	$$
	|u(x) - u( \d^i_n R^i_n /|x| \, x)| \leq C \log \frac{|x|}{\d^i_n R^i_n}; \qquad 
	|x| \in  [ R^i_n \d ^i_n, t_i^n ]. 
	$$
	The function  $u_n$, where it is positive,  can be estimated from above by its exponential. Where it is negative, 
	the previous formula allows to estimate it in absolute value by $C \log |x|$.  
	The last two formulas and the latter argument imply that the integrals on the
	l.h.s. of \eqref{eq:int-parts-bdry} converge to zero, giving a contradiction. 
	
	Assuming next that there is no $x^j_n \neq x^i_n$ as above, i.e. that there is 
	only one bubble, we choose the upper bound for $t_i$ (in the formula after \eqref{eq:claim-blow-up}) 
	to be a small but fixed number 
	$\tilde{r} > 0$. In this way, all the above arguments hold true, 
	with the exception that we need to replace the r.h.s. of \eqref{eq:int-parts-bdry} 
	by $- \d \pi + O(\tilde{r}^2)  + o_n(1)$. It is then sufficient to choose $\tilde{r}$ 
	small compared to $\d$ to reach again a contradiction.

	\
	
	\noindent {\bf Step 5: global radial decay.} This step can be skipped if in the 
	previous one only one bubble appears. If we have more bubbles instead, 
	we assume for simplicity that there are only two of them: $x^1_n$ and $x^2_n$. The general case 
	can be dealt with in a similar way, properly grouping points in clusters.

	By the previous step, fixing any small $\d >0$, we have the inequality 
	\eqref{eq:claim-blow-up}, and similarly 
	$$
	\overline{u}^2_n(s)  \leq  \overline{u}^2_n(R^2_n \d^2_n)  - (4-\d) \left( \log s - \log (R^2_n \d^2_n) \right), 
	$$
	for all $s \in [2 R^2_n \d^2_n, 1/8  |x^1_n- x^2_n|]$.

	By the Harnack inequality in Step 3 we have that the differences of the radial averages 
	$|\overline{u}^2_n(1/8  |x^1_n- x^2_n|) 
	- \overline{u}^1_n(1/8  |x^1_n- x^2_n|) |$ are uniformly bounded and  that moreover 
	(using also \eqref{eq:ineq-exp}) 
	$$
	|u_n(x)  - \overline{u}^1_n(1/8  |x^1_n- x^2_n|) | \leq C 
	\quad \hbox{ for evey } x \in \Omega_n; 
	$$
	$$
	\int_{\Omega_n} e^{u_n} \to 0; \qquad \oint_{\Xi_n} e^{u_n/2} \to 0, 
	$$
	where 
	$$
	\Omega_n = B^+_{x^1_n}( 8  |x^1_n- x^2_n|) \setminus \left( 
	B^+_{x^1_n }(1/8  |x^1_n-x^2_n|)  \cup B^+_{x^2_n}(1/8  |x^1_n- x^2_n|) \right),
	$$
	$$
		\Xi_n = \G^+_{x^1_n}( 8  |x^1_n- x^2_n|) \setminus \left( 
		\G^+_{x^1_n }(1/8  |x^1_n-x^2_n|)  \cup \G^+_{x^2_n}(1/8  |x^1_n- x^2_n|) \right). 
		$$
Integrating 
	\eqref{ecua-compact} and reasoning as in the previous step it follows that 
	$$
	\oint_{\partial^+ B_{x^1_n}(8 |x^1_n - x^2_n|)} 
	\frac{\partial u_n}{\partial \nu_n} \to - 8 \pi 
	\quad \hbox{ as } n \to + \infty, 
	$$
	since the total flux of the gradient over the boundary of $\Omega_n$ tends to zero. 
	One can then repeat the reasoning of Step 4 to show that 
	$$
	\overline{u}^1_n(s)  \leq  \overline{u}^1_n(8  |x^1_n- x^2_n|)  - 
	2 (4-\d) \left( \log s - \log (8  |x^1_n- x^2_n|) \right) + C 
	$$
	for all $s \in [8  |x^1_n- x^2_n|, \tilde{r}]$, where $C$ is a 
	fixed positive constant and $\tilde{r}$ is a small, but fixed, positive constant. 
	From these facts, one can pass to the next (and last) step. 
	Notice that these latter estimates imply that in the above set $\Omega_n$ 
	the integral of $e^{u_n}$ tends to zero (as well as the boundary integral of 
	$e^{u_n/2}$).

	\
	
	\noindent {\bf Step 6: conclusion.}  From Step 4 and Step 5 it follows that the integral 
	of $e^{u_n}$ (respectively, the boundary integral of $e^{u_n/2}$) are bounded in a fixed neighborhood of $p$, which is the first conclusion of the proposition. It readily follows also that $u_n$ diverges negatively on $\partial^+B_p(r)$.

\end{proof}

\begin{lemma}
	Under the conditions of Proposition \ref{p:6.3}, the asymptotic behavior given in \eqref{eq:profile-6} and \eqref{profile} hold. If moreover $\chi_n \leq 0$, then $S_1$ is empty.
\end{lemma}
	
\begin{proof}
	
	If $p \in S_0$, the statement follows from Case 1 in the proof of Theorem \ref{Ekel}. Recall that if $\mathfrak{D}(p)=1$, the only solution of \eqref{ecualimit1} is given by \eqref{profile2} (see Section 4). In this regard, take also into account Remark \ref{bordo}.  
	
	If $p \in S_1$, our main concern is to rule out the blow-up cluster phenomena.  Via a conformal map, we can work 
	in $B^+_0(r) \subset \R^2$, where $p$ is mapped to the origin. One has the distributional convergence 
		\begin{equation}\label{quant}
		(|K_n|e^{u_n})|_{B_p(r) \cap \S} \weakto \tilde m \beta \delta_{p} \quad \mbox{ and } \quad h_ne^{u_n/2}|_{\partial B_p(r) \cap \S} \weakto \tilde m (\beta+2\pi) \delta_{p},
		\end{equation}
where $\tilde m \in \mathbb{N}$ and
		
	\begin{equation}\label{eq:beta_i}
	\beta=2\pi \left (  \frac{h(p)}{\sqrt{h^2(p)+K(p)}} - 1 \right )
	\end{equation}

		In order to derive the result, we use a Pohozaev identity. Take $r>0$ such that $B^+_{0}(r)\cap S=\{0\}$ and apply Lemma~\ref{lema4} on $B_0^+(r)$ with $F(x)=x$ to obtain
		
		\begin{align*}
		\oint_{\partial B_0^+(r)} [4 K_n e^{u_n} ( x\cdot \nu) + 2 (\nabla u_n \cdot x)(\nabla u_n \cdot \nu) - |\nabla u_n|^2 x \cdot \nu] \\ 
		= \int_{B_0^+(r)} [ 4 \tilde{K}_n \nabla u_n \cdot x + 8  K_n e^{u_n} + (\nabla K_n \cdot x)e^{u_n}]. \nonumber
		\end{align*}
		
		Now, we split the boundary integrals into $\Gamma_0^+(r)$ and $\partial^+ B_0(r)$, to find:
		
		\begin{align} \label{quant0}
		\oint_{\Gamma_0^+(r)} \left[ 4 h_n e^{u_n/2} (\nabla u_n \cdot x) - 4 \tilde{h}_n (\nabla u_n \cdot x)   \right] +   \\  r \oint_{\partial^+ B_0(r)} \left[4 K_n e^{u_n}  + 2 \left(\frac{\nabla u_n \cdot x}{r}\right)^2  - |\nabla u_n|^2\right]  \nonumber \\ = \int_{B_0^+(r)} \left[4 \tilde{K}_n \nabla u_n \cdot x + 8 K_n e^{u_n} + (\nabla K_n \cdot x)e^{u_n}\right]. 	\nonumber \end{align}

		Taking into account that $|\nabla K_n|\leq C$ in $B_0^+(r)$, then
		
		\beq\label{quant1}
		\int_{B_0^+(r)} (\nabla K_n \cdot x)e^{u_n} \leq r \int_{B_0^+(r)} |\nabla K_n| e^{u_n} = O(r).
		\eeq
		
		On the other hand, integrating by parts
		
		$$
		\oint_{\Gamma_0^+(r)}4 h_n e^{u_n/2} (\nabla u_n \cdot x) = \Big[ 8h_n e^{u_n/2} x_1 \Big]_{-r}^r -8\oint_{\Gamma_0^+(r)} \nabla h_n \cdot x \, e^{u_n/2} - 8\oint_{\Gamma_0^+(r)} h_n e^{u_n/2}.
		$$
		Again, since $|\nabla h_n | \leq C$, we have
		
		\beq\label{quant2}
		\oint_{\Gamma_0^+(r)} \nabla h_n \cdot x \, e^{u_n/2} = O(r).
		\eeq
		
		Moreover, by Lemma \ref{key} one deduces
		
		\beq\label{quant3}
		\Big[ h_n e^{u_n/2} x_1 \Big]_{-r}^r \to 0, \quad \int_{\partial^+B_0(r)} K_ne^{u_n} \to 0, \quad \mbox{ as $n\to+\infty$}.
		\eeq
		
		Next, it is needed to estimate the gradient terms. In order to do it, we introduce the function $\displaystyle{v_n=u_n-\inf_{\partial^+ B_0(r)}u_n-w_n}$, where $v_n,w_n$ satisfy

		$$
		\left\{\begin{array}{ll}
		\displaystyle{-\Delta v_n + 2 \tilde{K}_n = 2 K_n e^{u_n}},  \qquad & \text{in $B_0^+(r)$,}\\
		\displaystyle{\frac{\partial v_n}{\partial n} +  2 \tilde{h}_n =2h_n e^{u_n/2}}, \qquad  &\text{on $\Gamma_0^+(r)$,} \\
		\displaystyle{v_n=0}, \qquad  &\text{on $\partial^+ B_0(r)$,}
		\end{array}\right.
		$$

		$$
		\left\{\begin{array}{ll}
		\displaystyle{\Delta w_n = 0},  \qquad & \text{in $B_0^+(r)$,}\\
		\displaystyle{\frac{\partial w_n}{\partial n} =0}, \qquad  &\text{on $\Gamma_0^+(r)$,} \\
		\displaystyle{w_n=u_n-\inf_{\partial^+ B_0(r)}u_n}, \qquad  &\text{on $\partial^+ B_0(r)$.}
		\end{array}\right.
		$$
		%

		Recall the Green's representation formula for $v_n$
		
		\begin{eqnarray} \label{eq:Green-local} \nonumber
			v_n(x)&=&\frac{-1}{\pi}\int_{B_0^+(r)} \log|x-y| \, (2 K_n(y) e^{u_n(y)} -2\tilde{K}_n(y)) dy \\ & - & \frac{1}{\pi}\oint_{\Gamma_0^+(r)} \log|x-y| \, \left( 2 h_n(y) e^{u_n(y)/2} - 2\tilde{h}_n(y) \right ) dy  + R_n(x),
		\end{eqnarray}
		where $R_n$ is uniformly bounded.
		
		Let us define the function
		
		$$
		\varphi_n=u_n-\inf_{B_0^+(r)} u_n\geq 0.
		$$
		
		By Lemma \eqref{u^-}, {\em c)},  we know that $\varphi_n\leq C$ on $\partial^+ B_0(r)$, so using the Green's representation formula for $\varphi_n$ we obtain that 
		
		\begin{eqnarray*}
			\varphi_n(x)&=&\frac{-1}{\pi}\int_{B_0^+(r)} \log|x-y| (2 K_n(y) e^{u_n(y)} -2\tilde{K}_n(y)) \, dy \\ & - & \frac{1}{\pi}\oint_{\Gamma_0^+(r)} \log|x-y| (2 h_n(y) e^{u_n(y)/2}- 2\tilde{h}_n(y)) \, dy \,  + \, O(1).
		\end{eqnarray*}

		Therefore
		
		$$
		\varphi_n\to\varphi= -4 \tilde m \log|x| + \phi \qquad \mbox{uniformly in $C^2(B_0^+(r)\setminus\{0\})$},
		$$
		where $\phi$ is a regular function on $B_0^+(r)$. As a consequence, we obtain that
		
		\beq\label{grad1}
		\nabla u_n = \nabla \varphi_n \to -4 \tilde m \frac{x}{|x|^2} + \nabla \phi.
		\eeq
		
		Now, by the boundedness of $\tilde{K}_n$, $\tilde{h}_n$ and \eqref{grad1} one gets
		
		\beq\label{quant4}
		\int_{B_0^+(r)} \tilde K_n \nabla u_n \cdot x = O(r), \qquad \int_{\Gamma_0^+(r)} \tilde h_n \nabla u_n \cdot x = O(r).
		\eeq
		
		By \eqref{quant1}, \eqref{quant2}, \eqref{quant3}, \eqref{quant4} and \eqref{grad1}, we can pass to the limit in \eqref{quant0} as $r\to0$ to conclude that $\tilde m=1$.
		
\		

We are now concerned with pointwise estimates for the blow--up profile, namely \eqref{profile}. For this it suffices to follow the arguments for the proof of Theorem~1.6 in \cite{bao}, which exploits the Green's representation formula  and a Pohozaev type identity, and also to use \eqref{C2conver} and the first bound in Lemma \ref{u^-} {\em c)}. Taking the limit problem \eqref{ecualimit1} into account, for the function $v_n$ defined in \eqref{scalesolution} one gets

$$
v_n(s,t)= 2 \log\left( \frac{2\lambda_0}{s^2+(t+t_0)^2-\lambda^2_0} \right) + O(1), \qquad \mbox{ in $B_0^+(r/\delta_n)$},
$$
where $t_0=\mathfrak{D}(0)\lambda_0$ and, since $v(0)=0$, $\lambda_0=\frac{2}{\mathfrak{D}^2(0)-1}$. In this way, taking into account Remark~\ref{bordo} and that $x_n \in \Gamma_0^+(r)$, the previous profile estimates can be recasted for $u_n$ to obtain \eqref{profile}, where $\lambda_n=\lambda_0\delta_n$. The gradient estimate \eqref{est1} follows from similar arguments.	
	
\		

Finally, as a consequence of \eqref{profile} and \eqref{est1}, we obtain that
	\begin{eqnarray*}
		\int_{B^+_{0}(r)}|\nabla u^-_n|^2 & \geq & 16 \int_{B^+_{0}(r)\setminus B_{0}^+(\delta_n\log\delta_n)} \frac{1}{|x|^2} + o\left(\frac{1}{|x|^2}\right) \\ & = & -4\log(\delta_n\log\delta_n) + o(\log(\delta_n\log\delta_n)).
	\end{eqnarray*}
	where $\delta_n$ is defined in \eqref{deltaeps}, so
	
	\beq\label{negativegradi}
	\int_{B^+_{0}(r)}|\nabla u^-_n|^2 \to +\infty.
	\eeq
	
	Obviously, \eqref{negativegradi} contradicts property {\em b)} in Lemma \ref{u^-} and we conclude the proof.

\end{proof}

We finish the proof of Proposition \ref{p:6.3} with the following two lemmas:

\begin{lemma}\label{l:D-final}
	For every $p \in S_0$, $\mathfrak{D}_\tau(p) = 0$. 
\end{lemma}

\begin{proof} If  $\mathfrak{D}_\tau(p) \neq 0$, with  $p \in S_0$, then $p$ is isolated in $S_0$. Recalling 
	that $S_1$ is finite by the previous lemma, $p$ is also isolated in $S$. 	By a conformal map we can pass to a problem in $B_p^+(r) \subset \R^2$, where $p$ is the unique singular point.
	
	Notice also that, at points in 
	$S_0$, the limit profiles of blowing-up solutions are given by formula \eqref{profile2}, and therefore the local accumulation of local volume must tend to $+ \infty$. This also implies that $\rho_n(u_n) \to + \infty$.

	We choose now $F(s,t)= (1,0)$ in Lemma \ref{lema4} and cut it off in a neighborhood of $p$. We can then continue with the proof of Proposition \ref{p:int-parts}, replacing the {\em global mass} $\rho_n$ with a {\em localized version}, namely $\int_{B^+_{p}(r)} e^{\hat{u}_n}$ for some small but fixed $r$. Take also into account Lemma \ref{u^-}, {\em c)}. In this way we still obtain $\mathfrak{D}_\tau(p) = 0$, as desired.

\end{proof}


\

\section{Appendix: test functions}
\setcounter{equation}{0}

Consider a point $p$  located on $\partial \Sigma$. Let $q=p+q_2 n(p)$, where $n$ is the outward normal vector to $\partial \Sigma$ and $q_2>0$ is small enough such that $q$ belongs to a regular extension of $\Sigma$. Given a parameter $\mu$ such that $\mu d(x,q)>1$ for every $x\in \Sigma$, we define the functions

$$
\varphi_{\mu,p}:\Sigma\to\R\qquad \varphi_{\mu,p}(x)=\log\frac{4\mu^2}{(\mu^2 d^2(x,q)-1)^2},
$$

\begin{equation}\label{testfunction}
\tilde{\varphi}_{\mu,p}:\Sigma\to\R\qquad \tilde{\varphi}_{\mu,p}(x)=\varphi_{\mu,p}-\log|K(x)|. 
\end{equation}

\begin{lemma}\label{l:test}
Let $p\in\partial\Sigma$ such that $\mathfrak{D}(p)>1$, let $I$ be as in \eqref{functional}, and let the function $\tilde{\varphi}_{\mu,p}$ be defined in \eqref{testfunction}. Then

$$
I(\tilde{\varphi}_{\mu,p})\to-\infty, \quad \oint_{\partial \Sigma} e^{\frac{\tilde{\varphi}_{\mu,p}}{2}} \to + \infty \quad \mbox{ as } \quad \mu\to \frac{1}{d(p,q)}=\frac{1}{q_2}.
$$
\end{lemma}

\begin{proof}

First of all, notice that

$$
I(\tilde{\varphi}_{\mu,p})=\int_{\Sigma} \left( \frac 1 2 |\nabla \tilde{\varphi}_{\mu,p}|^2 + 2 \tilde{K} \tilde{\varphi}_{\mu,p} + 2 e^{\varphi_{\mu,p}} \right) - 4 \oint_{\partial \Sigma} \mathfrak{D} e^{\varphi_{\mu,p}/2}.
$$

Letting $\varepsilon>0$, using the Young's inequality we can estimate the first term as 

$$
 \frac 1 2 \int_{\Sigma} |\nabla \tilde{\varphi}_{\mu,p}|^2 \leq  \left(\frac 1 2 +\varepsilon \right) \int_{\Sigma} |\nabla \varphi_{\mu,p}|^2 +O(1). 
$$

We consider each term of the functional $I_{\mu}$ separately, and claim that the following estimates hold:
\beq\label{testfunction1}
\int_{\Sigma} |\nabla \varphi_{\mu,p}|^2 \leq  \frac{8\pi }{\sqrt{\mu^2 q_2^2-1}} + o \left ( \frac{1 }{\sqrt{\mu^2 q_2^2-1}} \right) ,
\eeq

\beq\label{testfunction2}
\int_{\Sigma} e^{\varphi_{\mu,p}} \leq \frac{2\pi \mu q_2}{\sqrt{\mu^2 q_2^2-1}} +  o \left ( \frac{1 }{\sqrt{\mu^2 q_2^2-1}} \right) ,
\eeq

\beq\label{testfunction3}
\oint_{\partial \Sigma} \mathfrak{D} e^{\frac{\varphi_{\mu,p}}{2}} \geq \min_{B_p(r)\cap \partial \Sigma} \mathfrak{D} \frac{2\pi}{\sqrt{\mu^2 q_2^2-1}} +  o \left ( \frac{1 }{\sqrt{\mu^2 q_2^2-1}} \right) ,
\eeq

\beq\label{testfunction4}
\int_{\Sigma} \tilde{K} \, \tilde{\varphi}_{\mu,p} \leq O(1).
\eeq

From these, the assertion of Lemma \ref{l:test} follows immediately.

\

\underline{\textit{Proof of \eqref{testfunction1}}}.

\

Letting $r>q_2$, we divide the integral into two parts as follows
$$
\int_{\Sigma} |\nabla \varphi_{\mu,p}|^2 \, dV_{g} = \int_{\Sigma\setminus B_{q}(r)} |\nabla \varphi_{\mu,p}|^2 \, dV_{g} + \int_{B_{q}(r) \cap \Sigma} |\nabla \varphi_{\mu,p}|^2 \, dV_{g}.
$$

Applying the inequality $|\nabla d(x,z)^2|\leq 2 d(x,z)$, we have
$$
|\nabla \varphi_{\mu,p}(x)| = 2\mu^2 \frac{|\nabla d^2(x,q) |}{\mu^2 d^2(x,q)-1} \leq 4\mu^2 \frac{ d(x,q) }{\mu^2 d^2(x,q)-1}  \qquad \mbox{for every $ x\in \Sigma $}.
$$

By the use of the last inequality, normal coordinates centred at $q$ and the estimate

$$
\displaystyle{ dV_{g}=(1+o_{r}(1)) \, dx, \quad d(q,x)= |x-q|, \mbox{ for $x\in B_{q}(r)\cap \Sigma $}},
$$
with $r$ sufficiently small, one finds that

$$
\int_{B_{q}(r) \cap \Sigma} |\nabla \varphi_{\mu,p}|^2 \, dV_{g}  \leq 16\mu^{4} \int_{B^+_{q}(r)} (1+o_r(1)) \frac{|x-q|^{2} \, dx}{(\mu^2 |x-q|^{2}-1)^2},
$$
where $B^+_{q}(r)=\{(x_1,x_2)\in\mathbb{R}^2: x_1^2+(x_2+q_2)^2<r^2 \mbox{ and } x_2>0 \}$.

\

Now, writing $\mu^2|x-q|^2 = 1+o_r(1)$,  we can transform the integral

\beq\label{int0}
(A) :=\int_{B^+_{q}(r)} \frac{\mu^4|x-q|^{2} }{(\mu^2 |x-q|^{2}-1)^2} \, dx=\int_{B^+_{q}(r)} \frac{\mu^2(1+o_r(1)) }{(\mu^2 |x-q|^{2}-1)^2} \, dx.
\eeq

Next, take polar coordinates $(\rho,\theta)$ with center in $(0,-q_2)$ and consider the change of variable $t=\mu^2\rho^2-1$ to get

$$
\int_{B^+_{q}(r)} \frac{\mu^2 \, dx }{(\mu^2 |x-q|^{2}-1)^2}  = 2 \int^{\pi/2}_{\arcsin(\frac{q_2}{r})} \int_{\frac{q_2}{\sin \theta}}^r \frac{\mu^2 \rho }{(\mu^2\rho^2-1)^2}  \, d \rho \, d \theta
$$
$$
= \displaystyle{ \int^{\pi/2}_{\arcsin(\frac{q_2}{r})} \int_{\frac{\mu^2 q_2^2}{\sin^2 \theta}-1}^{\mu^2r^2-1} \frac{d t \, d \theta}{t^2}  =\int^{\pi/2}_{\arcsin(\frac{q_2}{r})} \frac{\sin^2 \theta}{\mu^2 q_2^2 \sin^2 \theta-\sin^2 \theta}  \, d\theta + O(1) }
$$
\beq\label{int1}
= \frac{\mu q_2}{\sqrt{\mu^2 q_2^2-1}} \left[ \arctan \left( \frac{\sqrt{\mu^2 q_2^2-1}}{\mu q_2} \tan \theta \right) \right]^{\pi/2}_{\arcsin(\frac{q_2}{r})} +O(1) =\frac{\mu \pi q_2}{2 \sqrt{\mu^2 q_2^2-1}} + O(1).
\eeq

\

Observe that for any $x\in\Sigma\setminus B_{q}(r)$ one has 

$$
|\nabla \varphi_{\mu,p}(x)| \leq \frac{diam(\Sigma)}{\mu^2r^2-1} \qquad \mbox{ for any } x\in\Sigma\setminus B_{q}(r).
$$

From this inequality, since $\mu r>\mu q_2 > 1$, we obtain that

\beq\label{int2}
\int_{\Sigma\setminus B_{q}(r)} |\nabla \varphi_{\mu,p}(x)|^2 = O(1).
\eeq

So, we have proved \eqref{testfunction1} by \eqref{int1} and \eqref{int2}.

\

\underline{\textit{Proof of \eqref{testfunction2}}}.

\

In order to compute the exponential terms, we divide the integral as

$$
\int_{\Sigma} e^{\varphi_{\mu,p}}=\int_{B_q(r)\cap\Sigma} e^{\varphi_{\mu,p}} + \int_{\Sigma\setminus B_q(r)} e^{\varphi_{\mu,p}}.
$$

We focus on the first integral. Again, using normal coordinates, we should compute

$$
\int_{B^+_q(r)} e^{\varphi_{\mu,p}} \, dx = 4 \int_{B^+_q(r)} \frac{ \mu^2  \, dx }{(\mu^2|x-q|^2-1)^2}=4 \, (A),$$
where $(A)$ is defined in \eqref{int0}.

Using the fact that

$$
e^{\varphi_{\mu,p}}\leq \frac{4 \mu^2}{(\mu^2r^2-1)^2} \qquad \mbox{ for any } x\in\Sigma\setminus B_{q}(r),
$$
then
\beq\label{int5}
\int_{\Sigma\setminus B_{q}(r)} e^{\varphi_{\mu,p}} = O(1).
\eeq

The estimates \eqref{int1} and \eqref{int5} conclude the proof.

\

\underline{\textit{Proof of \eqref{testfunction3}}}.

\

First, we split the integral as

$$
\oint_{\Sigma} \mathfrak{D} e^{\frac{\varphi_{\mu,p}}{2}}=\oint_{B_p(r)\cap\partial \Sigma} \mathfrak{D} e^{\frac{\varphi_{\mu,p}}{2}} + \oint_{\partial \Sigma\setminus B_p(r)} \mathfrak{D} e^{\frac{\varphi_{\mu,p}}{2}},
$$
and  focus on the first term. Next, 

$$
\oint_{B_p(r)\cap\partial \Sigma} \mathfrak{D} e^{\frac{\varphi_{\mu,p}}{2}} \geq \min_{B_p(r)\cap \partial \Sigma} \mathfrak{D}(x) \oint_{B_p(r)\cap\partial \Sigma} e^{\frac{\varphi_{\mu,p}}{2}}.
$$

Taking $r$ small enough and using normal coordinates, we compute

\beq\label{int6}
\oint_{B_p(r)\cap\partial \Sigma} e^{\frac{\varphi_{\mu,p}}{2}}=2 \oint_0^r\frac{\mu \, dx_1 }{\mu^2(x_1^2+q_2^2)-1} = \frac{\pi}{\sqrt{\mu^2q_2^2-1}} + O(1).
\eeq

In addition, we have 

$$
|e^{\frac{\varphi_{\mu,p}}{2}} | \leq O(1) \qquad \mbox{ in } \partial \Sigma\setminus B_{p}(r),
$$
so that

\beq\label{int7}
\left | \oint_{\partial \Sigma\setminus B_{p}(r)} \mathfrak{D} e^{\frac{\varphi_{\mu,p}}{2}} \right|  \leq O(1).
\eeq

The estimates \eqref{int6} and \eqref{int7} complete the proof of \eqref{testfunction3}.

\

\underline{\textit{Proof of \eqref{testfunction4}}}.

\

By the definition of $\tilde{\varphi}_{\mu,p}$, it is direct to check that 
$$
\int_{\Sigma} \tilde{K} \, \tilde{\varphi}_{\mu,p} \leq -C \int_{\Sigma} \varphi_{\mu,p} + O(1).
$$

Since
$$
\displaystyle{\log\frac{4\mu^2}{(\mu^2 q_2^2-1)^2} \leq \varphi_{\mu,p}}  \qquad \mbox{ for any } x\in \Sigma,
$$
the claim is proved.

\end{proof}

\end{document}